\pdfoutput=1
\RequirePackage{ifpdf}
\ifpdf 
\documentclass[pdftex]{sigma}
\else
\documentclass{sigma}
\fi

\usepackage{braket}

\newtheorem{thm}{Theorem}[section]
\newtheorem{prop}[thm]{Proposition}
\newtheorem{lem}[thm]{Lemma}
\newtheorem{cor}[thm]{Corollary}

\theoremstyle{definition}
\newtheorem{dfn}[thm]{Definition}
\newtheorem{rem}[thm]{Remark}

\numberwithin{equation}{section}

\DeclareMathOperator{\ev}{ev}

\DeclareMathOperator{\tr}{tr}
\DeclareMathOperator{\ad}{ad}
\DeclareMathOperator{\mult}{mult}

\newcommand{\ve}{\varepsilon}

\newcommand{\bbZ}{\mathbb{Z}}
\newcommand{\bbC}{\mathbb{C}}

\newcommand{\affY}{Y\big(\affsl\big)}

\newcommand{\frakg}{\mathfrak{g}}
\newcommand{\gl}{\mathfrak{gl}}
\newcommand{\fraksl}{\mathfrak{sl}}
\newcommand{\affsl}{\hat{\mathfrak{sl}}_N}
\newcommand{\bfid}{\mathbf{1}}

\begin{document}
\allowdisplaybreaks

\newcommand{\arXivNumber}{1805.01621}

\renewcommand{\PaperNumber}{020}

\FirstPageHeading

\ShortArticleName{Braid Group Action on Affine Yangian}

\ArticleName{Braid Group Action on Affine Yangian}

\Author{Ryosuke KODERA}

\AuthorNameForHeading{R.~Kodera}

\Address{Department of Mathematics, Graduate School of Science, Kobe University,\\ Kobe 657-8501, Japan}
\Email{\href{mailto:kodera@math.kobe-u.ac.jp}{kodera@math.kobe-u.ac.jp}}
\URLaddress{\url{http://www2.kobe-u.ac.jp/~kryosuke/}}

\ArticleDates{Received August 09, 2018, in final form February 27, 2019; Published online March 16, 2019}

\Abstract{We study braid group actions on Yangians associated with symmetrizable Kac--Moody Lie algebras. As an application, we focus on the affine Yangian of type A and use the action to prove that the image of the evaluation map contains the diagonal Heisenberg algebra inside $\hat{\mathfrak{gl}}_N$.}


\Classification{17B10; 17B37; 17B67}

\section{Introduction}

The purpose of this paper is to introduce braid group actions on Yangians and study their properties. Moreover we give an application for the affine Yangian of type~A.

Braid group actions on Kac--Moody Lie algebras and their integrable representations are classical. Lusztig initiated the study of braid group actions on quantized enveloping algebras associated with Kac--Moody Lie algebras. See~\cite[notes on Part~VI]{MR1227098} for historical remarks. They are fundamental tools for the construction of the PBW bases of the quantized enveloping algebras of finite type in~\cite{MR1013053, MR1066560} (see also \cite{MR1081014,MR1120927,MR1265471}). Then Lusztig~\cite{MR1035415} used the PBW basis for the construction of the canonical basis in ADE type. The braid group action is important also in affine type. Beck~\cite{MR1301623} used it to construct the Drinfeld generators and proved the equivalence of two presentations of the quantized affine algebra.

\looseness=1 Another important family of quantum groups are Yangians which were introduced by Drinfeld. The Yangian $Y(\frakg)$ associated with a simple Lie algebra $\frakg$ contains the universal enveloping algebra $U(\frakg)$ as a subalgebra and it is easy to see that the braid group action on~$U(\frakg)$ extends to~$Y(\frakg)$. Although this is a trivial observation, it seems that properties of the braid group action on the Yangian have not been seriously studied. (A different kind of braid group actions on quantized affine algebras and Yangians are studied by Ding--Khoroshkin~\cite{MR1745710}. It differs from the action introduced in this paper and our results are independent of their work.)

\looseness=1 Let us give a remark on a difference between the situations for the quantized enveloping algebra and the Yangian. The Yangian $Y(\frakg)$ is known to be an additive degeneration of the quantized enveloping algebra $U_q(\hat{\frakg})$ of affine type. The braid group acting on $U_q(\hat{\frakg})$ is of affine type, while one acting on $Y(\frakg)$ is of finite type. Thus the symmetry of Yangian coming from this consideration is smaller than that of quantized enveloping algebra given by Lusztig.

We can define the Yangian associated with a symmetrizable Kac--Moody Lie algebra $\frakg$. Then we see that the braid group of $\frakg$ acts on it. In this paper we mainly focus on the affine Yangian~$\affY$ of type A and the general case is studied in Appendix~\ref{appendixA}. The affine Yangian~$\affY$ is a two-parameter deformation of the universal enveloping algebra of the universal central extension of $\mathfrak{sl}_N\big[s,t^{\pm 1}\big]$, and is related to many interesting objects; symmetry of the spin Calogero--Sutherland model~\cite{MR3898327, MR1724950}; Schur--Weyl type duality for degenerate double affine Hecke algebra~\cite{MR2199856}; quiver variety associated with the cyclic quiver~\cite{MR3898327, MR1818101}.

The affine Yangian $\affY$ contains $U\big(\affsl\big)$ as a subalgebra and hence it admits an action of the braid group associated with $\affsl$. We give a formula for the action on generators of degree one. We use it to show the compatibility of the braid group action and the coproduct on $\affY$ introduced by Guay~\cite{MR2323534}, Guay--Nakajima--Wendlandt~\cite{MR3861718}. This compatibility result holds for the Yangian of arbitrary finite or affine type except for $A_1^{(1)}$ and $A_2^{(2)}$ as discussed in Appendix~\ref{appendixA}. See Propositions~\ref{prop:compatibility_coproduct} and~\ref{prop:compatibility_coproduct2} for precise statements.

We hope to apply the braid group action to the study of structure theory and representation theory of the affine Yangian. As a first step, we consider the following in the second half of the paper.

Guay~\cite{MR2323534} introduced an evaluation map for the affine Yangian $\affY$ whose target space is a certain completion of $U\big(\hat{\gl}_N\big)$. By its definition, the image contains~$U\big(\affsl\big)$, and we have expected that it contains $U\big(\hat{\gl}_N\big)$. We give an affirmative answer to this question, under a certain assumption on the parameters, by constructing elements of~$\affY$ whose images by the evaluation map coincide with generators of the diagonal Heisenberg algebra inside $\hat{\gl}_N$ (Theorem~\ref{thm:image}). To construct such elements, we use the braid group action. Certainly it is desirable to lift the Heisenberg subalgebra inside the affine Yangian and we will come back to this problem in a~future.

Our main result Theorem~\ref{thm:image} implies that the pull-back of an irreducible $\hat{\gl}_N$-module by the evaluation map is irreducible as a module of~$\affY$. We determine the highest weights of the evaluation modules in~\cite{kodera_evaluation}.

The plan of this paper is as follows. In Section~\ref{sec:affine_Yangian}, we define the affine Yangian $\affY$ and recall some automorphisms and the coproduct. In Section~\ref{section3}, we introduce the braid group action and study its properties. In particular, the compatibility with the coproduct is proved in Section~\ref{subsec:compatibility_coproduct}. We construct elements of $\affY$ whose images by the evaluation map coincide with Heisenberg generators in Section~\ref{section4}. In Appendix~\ref{appendixA}, we consider the braid group action on the Yangian associated with a symmetrizable Kac--Moody Lie algebra. Then we give a proof of the compatibility with the coproduct when it is known to be well-defined.

\section{Affine Yangian}\label{sec:affine_Yangian}

\subsection{Affine Yangian}

Fix an integer $N \geq 3$ throughout the paper. We use the notation $\{x,y\}=xy+yx$.

\begin{dfn}\label{dfn:Yangian}The affine Yangian $\affY$ is the algebra over $\bbC$ generated by $x_{i,r}^{+}$, $x_{i,r}^{-}$, $h_{i,r}$ ($i \in \mathbb{Z} / N\mathbb{Z}$, $r \in \mathbb{Z}_{\geq 0}$) with parameters $\ve_1, \ve_2 \in \bbC$ subject to the relations:
\begin{gather*}
[h_{i,r}, h_{j,s}] = 0, \qquad \big[x_{i,r}^{+}, x_{j,s}^{-}\big] = \delta_{ij} h_{i, r+s}, \qquad \big[h_{i,0}, x_{j,r}^{\pm}\big] = \pm a_{ij} x_{j,r}^{\pm},\\
\big[h_{i, r+1}, x_{j, s}^{\pm}\big] - \big[h_{i, r}, x_{j, s+1}^{\pm}\big]= \pm a_{ij} \dfrac{\varepsilon_1 + \varepsilon_2}{2} \big\{h_{i, r}, x_{j, s}^{\pm}\big\} - m_{ij} \dfrac{\varepsilon_1 - \varepsilon_2}{2} \big[h_{i, r}, x_{j, s}^{\pm}\big],\\
\big[x_{i, r+1}^{\pm}, x_{j, s}^{\pm}\big] - \big[x_{i, r}^{\pm}, x_{j, s+1}^{\pm}\big]= \pm a_{ij}\dfrac{\varepsilon_1 + \varepsilon_2}{2} \big\{x_{i, r}^{\pm}, x_{j, s}^{\pm}\big\}
- m_{ij} \dfrac{\varepsilon_1 - \varepsilon_2}{2} \big[x_{i, r}^{\pm}, x_{j, s}^{\pm}\big],\\
\sum_{w \in \mathfrak{S}_{1 - a_{ij}}}\big[x_{i,r_{w(1)}}^{\pm}, \big[x_{i,r_{w(2)}}^{\pm}, \dots, \big[x_{i,r_{w(1 - a_{ij})}}^{\pm}, x_{j,s}^{\pm}\big]\dots\big]\big] = 0, \qquad i \neq j,
\end{gather*}
where
\begin{gather*}
a_{ij} =
\begin{cases}
\hphantom{-}2 &\text{if } i=j, \\
-1 &\text{if } i=j \pm 1, \\
\hphantom{-}0 &\text{otherwise,}
\end{cases}\qquad
m_{ij} =
\begin{cases}
\hphantom{-}1 &\text{if } j=i-1, \\
-1 &\text{if } j=i+1, \\
\hphantom{-}0 &\text{otherwise.}
\end{cases}
\end{gather*}
\end{dfn}

\begin{rem}In this paper, we define $\hat\fraksl_N = \fraksl_N \otimes \bbC\big[t,t^{-1}\big] \oplus \bbC c$ without the degree operator. In \cite{MR3861718}, the Yangian $Y_{\hbar}(\frakg)$ of affine type is defined to be an algebra containing the degree opera\-tor~$d$ and one without~$d$ is denoted by~$Y_{\hbar}(\frakg')$. In particular, the algebra defined in Definition~\ref{dfn:Yangian} coincides with $Y_{\hbar,\ve}(\frakg')$ in the notation of \cite[Definition~7.1]{MR3861718}.
\end{rem}

The subalgebra generated by $x_{i,0}^+$, $x_{i,0}^-$, $h_{i,0}$ ($i \in \bbZ / N\bbZ$) is isomorphic to $U\big(\affsl\big)$ (See \cite[Theorem~6.1]{MR2323534} for $N \geq 4$ and \cite[Theorem~6.9]{GRW} in general). We write $x_{i}^{\pm} = x_{i,0}^{\pm}$, $h_{i} = h_{i,0}$ and identify them with the standard Chevalley generators of $\affsl$. Let $\{\alpha_i\}_{i \in \bbZ/N\bbZ}$ be the simple roots of~$\affsl$. The null root $\delta$ is given by $\delta=\sum\limits_{i=0}^{N-1} \alpha_i$. Let $\theta = \sum\limits_{i=1}^{N-1} \alpha_i$ be the highest root of~$\mathfrak{sl}_N$ and~$h_{-\theta}$ the coroot corresponding to $-\theta$. We denote by $\hat{\Delta}$, $\hat{\Delta}_+$, and $\hat{\Delta}_+^{\mathrm{re}}$ the set of roots, positive roots, and positive real roots for $\affsl$, respectively. We need to consider $\affsl \oplus \bbC d$ with the degree operator to deal with the coproduct on~$\affY$. Fix a nondegenerate invariant symmetric bilinear form $(\, ,\, )$ on $\affsl \oplus \bbC d$ such that $(x_i^+,x_i^-)=1$ and denote the induced bilinear form on the dual of the Cartan subalgebra by the same letter.

Let $\hat{W}$ be the Weyl group of $\affsl$ generated by the simple reflections $s_i$ ($i \in \bbZ/N\bbZ$). We denote by~$s_{\theta}$ the reflection corresponding to the highest root $\theta$. For each $\alpha \in \hat{\Delta}$, we assign the translation element $t_{\alpha}$ in $\hat{W}$. For example, we have $s_0 s_{\theta} = t_{\theta}$. The action of $t_{\alpha}$ on the root lattice $Q=\sum\limits_{i=1}^{N-1} \bbZ \alpha_i$ of $\mathfrak{sl}_N$ is given by $t_{\alpha}(\lambda) = \lambda - (\alpha,\lambda)\delta$ for $\lambda \in Q$.

Set $\hbar = \ve_1 + \ve_2$ and $\tilde{h}_{i,1} = h_{i,1} - \frac{\hbar}{2} h_{i}^2$. We can deduce the following identities directly from the defining relations of $\affY$.
\begin{lem}\label{lem:relations}
We have
\begin{gather*}
\big[h_{i,1}, x_{i}^{\pm}\big] = \pm2 x_{i,1}^{\pm} \pm \hbar \big\{h_i,x_i^{\pm}\big\}, \qquad \big[\tilde{h}_{i,1}, x_{i}^{\pm}\big] = \pm2 x_{i,1}^{\pm}, \qquad \big[x_{i,1}^{\pm},x_i^{\pm}\big] = \pm \hbar \big(x_i^{\pm}\big)^2,\\
\big[h_{i,1}, x_{j}^{\pm}\big] = \mp \left( x_{j,1}^{\pm} + \dfrac{\hbar}{2} \big\{h_i,x_j^{\pm}\big\} - m_{ij} \dfrac{\ve_1-\ve_2}{2}x_j^{\pm} \right) \qquad \text{if} \quad i=j \pm 1.
\end{gather*}
\end{lem}

\subsection{Automorphisms}

We introduce algebra (anti-)automorphisms of $\affY$ which will be used later.

Let $\omega$ be the algebra anti-automorphism of $\affY$ defined by $x_{i,r}^{\pm} \mapsto x_{i,r}^{\mp}$, $h_{i,r} \mapsto h_{i,r}$. Here anti-automorphism means that the invertible operator $\omega$ satisfies $\omega(XY)=\omega(Y)\omega(X)$. It is easy to check that it is well-defined. The following algebra automorphism corresponds to a~rotation of the Dynkin diagram.
\begin{prop}[{\cite[Lemma~3.5]{MR2199856}}] The assignment
\begin{gather*}
x_{i,r}^{\pm} \mapsto \sum_{s=0}^r \dbinom{r}{s} \ve_2^{r-s} x_{i-1,s}^{\pm}, \qquad h_{i,r} \mapsto \sum_{s=0}^r \dbinom{r}{s} \ve_2^{r-s} h_{i-1,s}
\end{gather*}
gives an algebra automorphism $\rho$ of $\affY$.
\end{prop}
In particular we have $\rho(x_{i,1}^{\pm}) = x_{i-1,1}^{\pm} + \ve_2 x_{i-1}^{\pm}$.
\begin{rem}The relation between generators $X_{i,r}^{\pm}$, $H_{i,r}$ with parameters $\lambda$, $\beta$ used in~\cite{MR2199856,MR2323534} and ours are as follows:
\begin{gather*}
X_{0,r}^{\pm} = \sum_{s=0}^r \dbinom{r}{s} \left( \dfrac{N}{4}(\ve_1-\ve_2)\right)^{r-s} x_{0,s}^{\pm},\qquad H_{0,r} = \sum_{s=0}^r \dbinom{r}{s} \left( \dfrac{N}{4}(\ve_1-\ve_2)\right)^{r-s} h_{0,s},\\
X_{i,r}^{\pm} = \sum_{s=0}^r \dbinom{r}{s} \left( \dfrac{i}{2}(\ve_1-\ve_2)\right)^{r-s} x_{i,s}^{\pm}, \qquad H_{i,r} = \sum_{s=0}^r \dbinom{r}{s} \left( \dfrac{i}{2}(\ve_1-\ve_2)\right)^{r-s} h_{i,s}, \qquad i \neq 0,\\
\lambda = \hbar, \qquad \beta = \frac{1}{2} \hbar - \dfrac{N}{4}(\ve_1-\ve_2).
\end{gather*}
\end{rem}

\subsection{Coproduct}

A formula for the coproduct on the affine Yangian $\affY$ was stated in \cite{MR2323534}. Guay--Nakajima--Wendlandt~\cite{MR3861718} gave a detailed proof of the well-definedness. To recall it, we consider a bigger algebra $Y\big(\affsl \oplus \bbC d\big)$, which is generated by $x_{i,r}^{\pm}$, $h_{i,r}$, $d$ with defining relations given in \cite[equation~(2.8)]{MR3861718}. Moreover we need certain completions $Y\big(\affsl \oplus \bbC d\big) \widehat{\otimes} Y\big(\affsl \oplus \bbC d\big)$ and $\affY \widehat{\otimes} \affY$ of the tensor products since the coproduct involves infinite sums. See \cite[Section~5]{MR3861718} for the precise definition of the completion.

We define the half Casimir operator $\Omega_+$ as follows. Let $\{u_k\}$ be a $\bbC$-basis of the Cartan subalgebra of $\affsl \oplus \bbC d$ and $\big\{u^k\big\}$ its dual basis with respect to the nondegenerate bilinear form~$(\, ,\, )$. Let $\big\{x_{\alpha}^{(k)}\big\}_{\alpha \in \hat{\Delta}, 1 \leq k \leq \mult \alpha}$ be a root vector basis satisfying $\big(x_{\alpha}^{(k)},x_{-\alpha}^{(l)}\big) = \delta_{k,l}$. Here $\mult \alpha$ denotes the dimension of the root space corresponding to~$\alpha$. We take simple root vectors as $x_{\pm \alpha_i}^{(1)} = x_i^{\pm}$. Put
\begin{gather*}
\Omega_+ = \sum_{k} u^k \otimes u_k + \sum_{\substack{\alpha \in \hat{\Delta}_+ \\ 1 \leq k \leq \mult \alpha}} x_{-\alpha}^{(k)} \otimes x_{\alpha}^{(k)}
\end{gather*}
as an element of $Y\big(\affsl \oplus \bbC d\big) \widehat{\otimes} Y\big(\affsl \oplus \bbC d\big)$.

Define a $\bbC$-linear operator $\square$ on $\affY$ by $\square(X) = X \otimes 1 + 1 \otimes X$. Note that $\square$ is not an algebra homomorphism, but satisfies $\square([X,Y]) = [\square(X), \square(Y)]$.
\begin{thm}[{\cite[Definition~4.6, Theorem~4.9, Proposition~5.18, Section~7]{MR3861718}}]\label{thm:definition_of_coproduct} There exists an algebra homomorphism $\Delta \colon \affY \to \affY \widehat{\otimes} \affY$ uniquely determined by
\begin{gather*}
\Delta(X) = \square(X) \qquad \text{for} \quad X=x_{i}^{\pm}, h_i,\\
\Delta\big(x_{i,1}^+\big) = \square(x_{i,1}^+) - \hbar \big[1 \otimes x_i^+, \Omega_+\big],\\
\Delta\big(x_{i,1}^-\big) = \square(x_{i,1}^-) + \hbar \big[x_i^- \otimes 1, \Omega_+\big],\\
\Delta\big(\tilde{h}_{i,1}\big)= \square(\tilde{h}_{i,1}) + \hbar [h_i \otimes 1, \Omega_+].
\end{gather*}
\end{thm}
\begin{rem}Explicitly we have
\begin{gather*}
[1 \otimes x_i^+, \Omega_+] = -h_i \otimes x_i^+ + \sum_{\substack{\alpha \in \hat{\Delta}_+\\ 1 \leq k \leq \mult \alpha}} x_{-\alpha}^{(k)} \otimes \big[x_i^+, x_{\alpha}^{(k)}\big],
\\
[x_i^- \otimes 1, \Omega_+] = x_i^- \otimes h_i + \sum_{\substack{\alpha \in \hat{\Delta}_+\\ 1 \leq k \leq \mult \alpha}} \big[x_i^-, x_{-\alpha}^{(k)}\big] \otimes x_{\alpha}^{(k)},
\\
[h_i \otimes 1, \Omega_+] = - \sum_{\alpha \in \hat{\Delta}_+^{\mathrm{re}}} (\alpha_i,\alpha)\, x_{-\alpha}^{(1)} \otimes x_{\alpha}^{(1)}.
\end{gather*}
Hence the target of $\Delta$ is $\affY \widehat{\otimes} \affY$ without~$d$.
\end{rem}

\section{Braid group action}\label{section3}

We define automorphisms $T_i$ $(i \in \bbZ/N\bbZ)$ of the affine Yangian and study their properties.

\subsection{Definition}

Since the adjoint actions of $x_i^{\pm}$ on $\affY$ are locally nilpotent, the operators $\exp\ad x_i^{\pm}$ are well defined by
\begin{gather*}
\exp\ad x_i^{\pm} = \sum_{n=0}^{\infty} \dfrac{1}{n!} \big(\ad x_i^{\pm}\big)^n.
\end{gather*}
These are algebra automorphisms of $\affY$ as $\ad x_i^{\pm}$ are derivations. We define an algebra automorphism $T_i$ of $\affY$ for each $i \in \bbZ/N\bbZ$ by
\begin{gather*}
T_i = \exp\ad x_i^{+} \exp\ad (-x_i^{-}) \exp\ad x_i^{+}.
\end{gather*}
This operator appears in \cite{MR3861718} and is used to construct real root vectors of Yangians.

\subsection{Braid relations}

A proof of the following proposition is exactly the same as one for the fact that $\{T_i\}$ satisfy the braid relations as automorphisms of $U\big(\affsl\big)$. We give a proof for the sake of completeness.

\begin{prop}\label{prop:braid}The operators $\{T_i\}$ satisfy the braid relations. That is, we have
\begin{gather*}
T_i T_j = T_j T_i \quad \text{if $a_{ij}=0$},\qquad T_i T_j T_i = T_j T_i T_j \quad \text{if $a_{ij}=-1$}.
\end{gather*}
\end{prop}
For each $w \in \hat{W}$ with a reduced expression $w=s_{i_1} \cdots s_{i_l}$, we can define an algebra automorphism $T_w$ of $\affY$ by $T_w = T_{i_1} \cdots T_{i_l}$ thanks to the braid relations.

Let us start the proof with some preparations. A proof of the following formulas is straightforward.
\begin{lem}\label{lem:expad1} Assume $a_{ij}=-1$. Then
\begin{enumerate}\itemsep=0pt
\item $\exp\ad x_{i}^{+}$ sends:
\begin{gather*}
x_i^+ \mapsto x_i^+, \qquad x_j^+ \mapsto x_j^+ + \big[x_i^+,x_j^+\big],\qquad x_i^- \mapsto x_i^-+h_i-x_i^+, \qquad x_j^- \mapsto x_j^-,\\
h_i \mapsto h_i-2x_i^+,\qquad h_j \mapsto h_j+x_i^+,\qquad \big[x_i^+,x_j^+\big] \mapsto \big[x_i^+,x_j^+\big],\\
 \big[x_i^-,x_j^-\big] \mapsto \big[x_i^-,x_j^-\big] + x_j^-;
\end{gather*}
\item $\exp\ad (-x_{i}^{-})$ sends:
\begin{gather*}
x_i^+ \mapsto x_i^++h_i-x_i^-, \qquad x_j^+ \mapsto x_j^+,\qquad x_i^- \mapsto x_i^-, \qquad x_j^- \mapsto x_j^--[x_i^-,x_j^-],\\
h_i \mapsto h_i-2x_i^-,\qquad h_j \mapsto h_j+x_i^-,\qquad \big[x_i^+,x_j^+\big] \mapsto \big[x_i^+,x_j^+\big]-x_j^+,\\
 \big[x_i^-,x_j^-\big] \mapsto \big[x_i^-,x_j^-\big].
\end{gather*}
\end{enumerate}
\end{lem}

The following two propositions are well known. Proposition~\ref{prop:formula1} follows from Lemma~\ref{lem:expad1}. Then Proposition~\ref{prop:iji_pre} follows from Proposition~\ref{prop:formula1}.
\begin{prop}\label{prop:formula1} We have
\begin{gather*}
T_i(x_j^+) = \begin{cases}
-x_i^- & \text{if $i=j$},\\
\big[x_i^+,x_j^+\big] & \text{if $a_{ij}=-1$},\\
x_j^+ & \text{if $a_{ij}=0$},
\end{cases}\qquad
T_i\big(x_j^-\big) = \begin{cases}
-x_i^+ & \text{if $i=j$},\\
-\big[x_i^-,x_j^-\big] & \text{if $a_{ij}=-1$},\\
x_j^- & \text{if $a_{ij}=0$},
\end{cases}
\\
T_i(h_j) = \begin{cases}
-h_i & \text{if $i=j$},\\
h_j + h_i & \text{if $a_{ij}=-1$},\\
h_j & \text{if $a_{ij}=0$}.
\end{cases}
\end{gather*}
\end{prop}

\begin{prop}\label{prop:iji_pre}Assume $a_{ij}=-1$. Then we have
\begin{gather*}
T_i T_j \big(x_i^+\big) = x_j^+,\qquad T_i T_j \big(x_i^-\big) = x_j^-,\qquad T_i T_j (h_i) = h_j.
\end{gather*}
\end{prop}

\begin{proof}[Proof of Proposition~\ref{prop:braid}] We use that $\varphi \circ \exp\ad X \circ \varphi^{-1} = \exp \ad \varphi(X)$ holds for any algebra automorphism~$\varphi$. If $a_{ij}=0$, then we have
\begin{gather*}
T_i T_j T_i^{-1} = \exp\ad T_i\big(x_j^+\big) \exp\ad T_i\big({-}x_j^-\big) \exp\ad T_i\big(x_j^+\big) = T_j
\end{gather*}
by the formulas in Proposition~\ref{prop:formula1}. If $a_{ij}=-1$, then we have
\begin{gather*}
T_i T_j T_i T_j^{-1} T_i^{-1} = \exp\ad T_i T_j \big(x_i^+\big) \exp\ad T_i T_j \big({-}x_i^-\big) \exp\ad T_i T_j \big(x_i^+\big) = T_j
\end{gather*}
by the formulas in Proposition~\ref{prop:iji_pre}. The proof is complete.
\end{proof}

We give an alternative definition of $T_i$.
\begin{prop}The operator $T_i$ coincides with $\exp\ad \big({-}x_i^{-}\big) \exp\ad x_i^{+} \exp\ad \big({-}x_i^{-}\big)$.
\end{prop}
\begin{proof}Put $T_i' = \exp\ad \big({-}x_i^{-}\big) \exp\ad x_i^{+} \exp\ad \big({-}x_i^{-}\big)$ for a while. Then we have
\begin{gather*}
T_i T_i' T_i^{-1} = \exp\ad T_i\big({-}x_i^{-}\big) \exp\ad T_i\big(x_i^+\big) \exp\ad T_i\big({-}x_i^{-}\big) = T_i
\end{gather*}
by the formulas in Proposition~\ref{prop:formula1}. Hence the assertion is proved.
\end{proof}

Note that $T_i^{-1}$ is given by
\begin{gather*}
\exp\ad x_i^- \exp\ad \big({-}x_i^+\big) \exp\ad x_i^- = \exp\ad \big({-}x_i^+\big) \exp\ad x_i^- \exp\ad \big({-}x_i^+\big).
\end{gather*}

\begin{lem}\label{lem:rho_omega_and_T} We have $(i)$ $\omega \circ T_i = T_{i} \circ \omega$ and $(ii)$ $\rho \circ T_i = T_{i-1} \circ \rho$.
\end{lem}
\begin{proof}For the assertion (i), we use $\omega (\exp\ad X (Y)) = \exp\ad(-\omega(X))(\omega(Y))$. Then we have $\omega(T_i(X))=T_i'(\omega(X))$. Since we have proved $T_i'=T_i$, the assertion holds.

The assertion (ii) is obvious.
\end{proof}

Let $M$ be a $\affY$-module and assume that $x_i^{\pm}$ acts on $M$ locally nilpotently. Then an automorphism $T_i^M$ of $M$ is defined similarly. The following property is immediate.
\begin{prop}Let $M$ be a $\affY$-module and assume that $x_i^{\pm}$ acts on $M$ locally nilpotently. Then for any $X \in \affY$ and $m \in M$,
\begin{gather*}
T_i^M(Xm)=T_i(X)T_i^M(m)
\end{gather*}
holds.
\end{prop}

In fact, this is a general property for any associative algebra such that it contains $U\big(\affsl\big)$ as a subalgebra and the operators $\exp\ad x_i^{\pm}$ are well-defined.

\subsection{Action on generators}

We compute the action of $T_i$ on $x_{j,1}^{\pm}$, $\tilde{h}_{j,1}$. We will use the following formulas.
\begin{lem}\label{lem:expad2}Assume $a_{ij}=-1$. Then
\begin{enumerate}\itemsep=0pt
\item $\exp\ad x_{i}^{+}$ sends:
\begin{gather*}
x_{i,1}^+ \mapsto x_{i,1}^+-\hbar \big(x_i^+\big)^2,\\
x_{i,1}^- \mapsto x_{i,1}^- + h_{i,1}-x_{i,1}^+ -\tfrac{\hbar}{2}\big\{ h_i, x_i^+\big\} + \hbar \big(x_i^+\big)^2,\\
x_{j,1}^- \mapsto x_{j,1}^-,\\
\big[x_i^-,x_{j,1}^-\big] \mapsto \big[x_i^-,x_{j,1}^-\big]+x_{j,1}^-,\\
\big\{h_i,x_i^+\big\} \mapsto \big\{h_i,x_i^+\big\} - 4\big(x_i^+\big)^2,\\
\big(x_i^+\big)^2 \mapsto (x_i^+)^2;
\end{gather*}
\item $\exp\ad (-x_{i}^{-})$ sends:
\begin{gather*}
x_{i,1}^+ \mapsto x_{i,1}^+ + h_{i,1}-x_{i,1}^- -\tfrac{\hbar}{2}\big\{ h_i, x_i^-\big\} + \hbar \big(x_i^-\big)^2,\\
x_{i,1}^- \mapsto x_{i,1}^- - \hbar\big(x_i^-\big)^2,\\
x_{j,1}^- \mapsto x_{j,1}^- - \big[x_i^-,x_{j,1}^-\big],\\
h_{i,1} \mapsto h_{i,1} - 2x_{i,1}^- - \hbar\big\{ h_i, x_i^- \big\} + 3\hbar \big(x_i^-\big)^2,\\
\big\{h_i,x_i^+\big\} \mapsto \big\{h_i,x_i^+\big\} - 3\big\{h_i,x_i^-\big\} - 2\big\{x_i^+,x_i^-\big\} + 2h_i^2 + 4\big(x_i^-\big)^2,\\
\big(x_i^+\big)^2 \mapsto \big(x_i^+\big)^2 + h_i^2 + \big(x_i^-\big)^2+ \big\{h_i,x_i^+\big\} - \big\{h_i,x_i^-\big\} - \big\{x_i^+,x_i^-\big\}.
\end{gather*}
\end{enumerate}
\end{lem}
\begin{proof}By a direct computation using the defining relations of $\affY$ with Lemma~\ref{lem:relations}. For example, by
\begin{gather*}
\ad x_i^{+} \big(x_{i,1}^-\big) = h_{i,1}, \\
\big(\ad x_i^{+}\big)^2 \big(x_{i,1}^-\big) = \big[x_i^+,h_{i,1}\big] = -2x_{i,1}^+ - \hbar \big\{h_i,x_i^+\big\}, \\
\big(\ad x_i^{+}\big)^3 \big(x_{i,1}^-\big) = -2 \big[x_i^+, x_{i,1}^+\big] - \hbar \big\{\big[x_i^+,h_i\big],x_i^+\big\} = 2\hbar\big(x_i^+\big)^2 + 4\hbar \big(x_i^+\big)^2 = 6\hbar\big(x_i^+\big)^2, \\
\big(\ad x_i^{+}\big)^4 \big(x_{i,1}^-\big)=0,
\end{gather*}
we have
\begin{gather*}
\exp\ad x_i^+ \big(x_{i,1}^-\big) = x_{i,1}^- + h_{i,1} - x_{i,1}^+ - \tfrac{\hbar}{2} \big\{h_i,x_i^+\big\} + \hbar\big(x_i^+\big)^2.\tag*{\qed}
\end{gather*}\renewcommand{\qed}{}
\end{proof}

\begin{prop}\label{prop:formula2}We have
\begin{gather*}
T_i\big(x_{j,1}^+\big) = \begin{cases}
-x_{i,1}^- + \tfrac{\hbar}{2}\big\{ h_i, x_i^- \big\} & \text{if $i=j$},\\
\big[x_i^+,x_{j,1}^+\big] & \text{if $a_{ij}=-1$},\\
x_{j,1}^+ & \text{if $a_{ij}=0$},
\end{cases} \\
T_i\big(x_{j,1}^-\big) = \begin{cases}
-x_{i,1}^+ + \tfrac{\hbar}{2}\big\{ h_i, x_i^+ \big\} & \text{if $i=j$},\\
-\big[x_{i}^-,x_{j,1}^-\big] & \text{if $a_{ij}=-1$},\\
x_{j,1}^- & \text{if $a_{ij}=0$},
\end{cases}
\\
T_i\big(\tilde{h}_{j,1}\big) = \begin{cases}
-\tilde{h}_{i,1} - \hbar \big\{ x_i^+,x_i^- \big\} & \text{if $i=j$},\\
\tilde{h}_{j,1} + \tilde{h}_{i,1} + \tfrac{\hbar}{2} \big\{ x_i^+,x_i^- \big\} + m_{ij}\tfrac{\ve_1 - \ve_2}{2} h_i & \text{if $a_{ij}=-1$},\\
\tilde{h}_{j,1} & \text{if $a_{ij}=0$}.
\end{cases}
\end{gather*}
\end{prop}

\begin{proof}The formulas for $a_{ij}=0$ trivially hold. Let us consider the other cases.

We compute $T_i\big(x_{j,1}^-\big)$. First assume $i=j$. We have
\begin{gather*}
\exp\ad \big({-}x_i^-\big)\exp\ad x_i^+\big(x_{i,1}^-\big) = \exp\ad \big({-}x_i^-\big)\big(x_{i,1}^- + h_{i,1}-x_{i,1}^+ -\tfrac{\hbar}{2}\big\{ h_i, x_i^+\big\} + \hbar \big(x_i^+\big)^2\big)\\
\qquad {}= \big(x_{i,1}^- - \hbar(x_i^-)^2\big)+\big(h_{i,1}-2x_{i,1}^- - \hbar\big\{ h_i, x_i^- \big\} + 3\hbar \big(x_i^-\big)^2\big)\\
\qquad\quad{}-\big(x_{i,1}^+ + h_{i,1}-x_{i,1}^--\tfrac{\hbar}{2}\big\{ h_i, x_i^-\big\} + \hbar \big(x_i^-\big)^2\big)\\
\qquad\quad{}-\tfrac{\hbar}{2}\big(\big\{h_i,x_i^+\big\} - 3\big\{h_i,x_i^-\big\} - 2\big\{x_i^+,x_i^-\big\} + 2h_i^2 + 4\big(x_i^-\big)^2\big)\\
\qquad\quad{} +\hbar\big(\big(x_i^+\big)^2 + h_i^2 + \big(x_i^-\big)^2+ \big\{h_i,x_i^+\big\} - \big\{h_i,x_i^-\big\} - \big\{x_i^+,x_i^-\big\}\big)\\
\qquad{}=-x_{i,1}^+ + \tfrac{\hbar}{2}\big\{h_i,x_i^+\big\} +\hbar \big(x_i^+\big)^2.
\end{gather*}
Then we have
\begin{gather*}
\exp\ad x_i^+\big({-}x_{i,1}^+ + \tfrac{\hbar}{2}\big\{h_i,x_i^+\big\} +\hbar \big(x_i^+\big)^2\big) \\
\qquad {}= -\big( x_{i,1}^+-\hbar \big(x_i^+\big)^2 \big) + \tfrac{\hbar}{2}\big( \big\{h_i,x_i^+\big\} - 4\big(x_i^+\big)^2 \big) + \hbar \big(x_i^+\big)^2
 =-x_{i,1}^++\tfrac{\hbar}{2}\big\{h_i,x_i^+\big\}.
\end{gather*}
Next assume $a_{ij}=-1$. Then we have
\begin{gather*}
T_i\big(x_{j,1}^-\big)=\exp\ad x_i^+\exp\ad \big({-}x_i^-\big)\big(x_{j,1}^-\big) = \exp\ad x_i^+\big(x_{j,1}^- - \big[x_i^-,x_{j,1}^-\big]\big)\\
\hphantom{T_i\big(x_{j,1}^-\big)}{} = x_{j,1}^- - \big(\big[x_i^-,x_{j,1}^-\big]+x_{j,1}^-)=-\big[x_i^-,x_{j,1}^-\big]. 
\end{gather*}
The formulas for $T_i\big(x_{j,1}^+\big)$ are obtained by applying $\omega$ to these results.

We compute $T_{i}(h_{j,1})$. First assume $i=j$. Then we have
\begin{gather*}
T_i(h_{i,1}) = T_i\big(\big[x_{i}^+,x_{i,1}^-\big]\big) = \big[{-}x_{i}^-, -x_{i,1}^+ + \tfrac{\hbar}{2}\big\{ h_i, x_i^+ \big\}\big] = -h_{i,1} - \hbar \big\{x_i^+,x_i^-\big\} + \hbar h_i^2.
\end{gather*}
Next assume $a_{ij}=-1$. Then we have
\begin{gather*}
T_i(h_{j,1}) = T_i\big(\big[x_{j}^+,x_{j,1}^-\big]\big) = \big[\big[x_i^+,x_{j}^+\big], -\big[x_i^-,x_{j,1}^-\big]\big] \\
\hphantom{T_i(h_{j,1})}{} = - \big[\big[\big[x_i^+,x_j^+\big],x_i^-\big],x_{j,1}^-\big]-\big[x_i^-,\big[\big[x_i^+,x_j^+\big],x_{j,1}^-\big]\big].
\end{gather*}
Since we have
\begin{gather*}
\big[\big[\big[x_i^+,x_j^+\big],x_i^-\big],x_{j,1}^-\big]=\big[\big[h_i,x_j^+\big],x_{j,1}^-\big]=\big[{-}x_j^+,x_{j,1}^-\big]=-h_{j,1}
\end{gather*}
and
\begin{gather*}
\big[x_i^-,\big[[x_i^+,x_j^+\big],x_{j,1}^-\big]\big] =\big[x_i^-,\big[x_i^+,h_{j,1}\big]\big] = \big[x_i^-,x_{i,1}^++\tfrac{\hbar}{2}\big\{h_j,x_i^+\big\}-m_{ji}\tfrac{\ve_1 - \ve_2}{2} x_i^+\big] \\
\hphantom{\big[x_i^-,\big[[x_i^+,x_j^+\big],x_{j,1}^-\big]\big]}{} = -h_{i,1} + \tfrac{\hbar}{2}\big( \big\{\big[x_i^-,h_j\big],x_i^+\big\} + \big\{ h_j, \big[x_i^-, x_i^+\big] \big\}\big) + m_{ji}\tfrac{\ve_1 - \ve_2}{2} h_i \\
\hphantom{\big[x_i^-,\big[[x_i^+,x_j^+\big],x_{j,1}^-\big]\big]}{} = -h_{i,1} - \tfrac{\hbar}{2} \big\{x_i^-,x_i^+\big\} - \hbar h_i h_j + \tfrac{\ve_1 - \ve_2}{2}m_{ji} h_i,
\end{gather*}
we conclude
\begin{gather*}
T_i(h_{j,1})=h_{j,1}+h_{i,1}+\tfrac{\hbar}{2} \big\{x_i^+,x_i^-\big\} + \hbar h_i h_j + \tfrac{\ve_1 - \ve_2}{2}m_{ij} h_i.
\end{gather*}
We can easily obtain the formulas for $T_{i}(\tilde{h}_{j,1})$ from these identities.
\end{proof}

\begin{prop}\label{prop:iji}Assume $a_{ij}=-1$. Then we have
\begin{gather*}
T_i T_j \big(x_{i,1}^+\big) = x_{j,1}^+ - m_{ij}\tfrac{\ve_1 - \ve_2}{2}x_j^+ - \tfrac{\hbar}{2}\big\{T_i\big(x_j^+\big),x_i^-\big\}, \\
T_i T_j \big(x_{i,1}^-\big) = x_{j,1}^- - m_{ij}\tfrac{\ve_1 - \ve_2}{2}x_j^- - \tfrac{\hbar}{2}\big\{x_i^+,T_i\big(x_j^-\big)\big\}, \\
T_i T_j (h_{i,1}) = h_{j,1} - m_{ij}\tfrac{\ve_1 - \ve_2}{2}h_j - \tfrac{\hbar}{2}\{x_i^+,x_i^-\} + \tfrac{\hbar}{2}\big\{T_i(x_j^+),T_i\big(x_j^-\big)\big\}.
\end{gather*}
\end{prop}
\begin{proof}We show the first identity. Since we have
\begin{gather*}
T_iT_j\big(x_{i,1}^+\big) = T_i\big(\big[x_{j}^+,x_{i,1}^+\big]\big) = \big[\big[x_i^+,x_j^+\big],-x_{i,1}^-+\tfrac{\hbar}{2}\big\{h_i,x_i^{-}\big\}\big],
\end{gather*}
it follows from
\begin{gather*}
\big[\big[x_i^+,x_j^+\big],x_{i,1}^-\big] = \big[h_{i,1},x_j^+\big] = -\big( x_{j,1} + \tfrac{\hbar}{2}\big\{h_i,x_j^{+}\big\} - m_{ij} \tfrac{\ve_1 - \ve_2}{2}x_j^+ \big)
\end{gather*}
and
\begin{gather*}
\big[\big[x_i^+,x_j^+\big],\big\{h_i,x_i^{-}\big\}\big]= \big\{\big[\big[x_i^+,x_j^+\big],h_i\big],x_i^-\big\} + \big\{h_i,\big[\big[x_i^+,x_j^+\big],x_i^-\big]\big\} \\
 \hphantom{\big[\big[x_i^+,x_j^+\big],\big\{h_i,x_i^{-}\big\}\big]}{} = -\big\{\big[x_i^+,x_j^+\big],x_i^-\big\} - \big\{h_i,x_j^{+}\big\} = -\big\{T_i(x_j^+),x_i^-\big\} - \big\{h_i,x_j^{+}\big\}.
\end{gather*}
The second identity is obtained by applying $\omega$ to the first one. The third identity follows from
\begin{gather*}
T_iT_j(h_{i,1})= T_iT_j\big(\big[x_{i,1}^+,x_{i}^-\big]\big) = \big[x_{j,1}^+-m_{ij}\tfrac{\ve_1-\ve_2}{2}x_j^+-\tfrac{\hbar}{2}\big\{T_i\big(x_j^+\big),x_i^-\big\},x_j^-\big]
\end{gather*}
and
\begin{gather*}
\big[\big\{T_i\big(x_j^+\big),x_i^-\big\},x_j^-\big] = \big\{\big[T_i(x_j^+),x_j^-\big],x_i^-\big\} + \big\{T_i\big(x_j^+\big),\big[x_i^-,x_j^-\big]\big\}\\
\hphantom{\big[\big\{T_i\big(x_j^+\big),x_i^-\big\},x_j^-\big]}{} = \big\{x_i^+,x_i^-\big\} - \big\{T_i\big(x_j^+\big),T_i\big(x_j^-\big)\big\}.\tag*{\qed}
\end{gather*}\renewcommand{\qed}{}
\end{proof}

We give formulas for $T_i^2$. Proofs are straightforward.
\begin{prop} We have $T_i^2(X)=X$ for $X=x_{j}^{\pm}, x_{j,1}^{\pm}$ $(a_{ij} \neq -1)$, $h_k, h_{k,1}$ $(k \in \bbZ/N\bbZ)$. We have $T_i^2(X)=-X$ for $X=x_{j}^{\pm}, x_{j,1}^{\pm}$ $(a_{ij} = -1)$. In particular, $T_i^4 = 1$.
\end{prop}

\subsection{Compatibility with the coproduct}\label{subsec:compatibility_coproduct}

The goal of this subsection is to prove the following proposition.
\begin{prop}\label{prop:compatibility_coproduct}We have $\Delta \circ T_i = (T_i \otimes T_i) \circ \Delta$.
\end{prop}

We use the following lemma to prove Proposition~\ref{prop:compatibility_coproduct}.

\begin{lem}\label{lem:Omega}We have $(T_i \otimes T_i)\Omega_+ = \Omega_+ + x_i^+ \otimes x_i^- - x_i^- \otimes x_i^+$.
\end{lem}
\begin{proof}Note that the bilinear form $(\, ,\, )$ is $T_i$-invariant. Therefore, if $\{u_k\}$ and $\big\{u^k\big\}$ are dual bases, then $\{T_i(u_k)\}$ and $\big\{T_i\big(u^k\big)\big\}$ are also dual bases. Moreover, if we put
\begin{gather*}
y_{\alpha}^{(k)} = \begin{cases}
T_i\big(x_{s_i(\alpha)}^{(k)}\big) & \text{if $\alpha \in \hat{\Delta} \setminus \{\alpha_i, -\alpha_i\}$},\\
x_i^+ & \text{if $\alpha = \alpha_i$},\\
x_i^- & \text{if $\alpha = -\alpha_i$},
\end{cases}
\end{gather*}
then $\big\{ y_{\alpha}^{(k)} \big\}$ also satisfy $\big(y_{\alpha}^{(k)}, y_{-\alpha}^{(l)}\big) = \delta_{k,l}$ and we have
\begin{gather*}
\big\{ T_i \big(x_{\alpha}^{(k)}\big) \big\}_{\alpha \in \hat{\Delta}_+, 1 \leq k \leq \mult \alpha} = \big\{y_\alpha^{(k)}\big\}_{\alpha \in \hat{\Delta}_+, 1 \leq k \leq \mult \alpha} \cup \big\{{-}x_i^-\big\} \setminus \big\{x_i^+\big\},\\
\big\{ T_i \big(x_{-\alpha}^{(k)}\big) \big\}_{\alpha \in \hat{\Delta}_+, 1 \leq k \leq \mult \alpha} = \big\{y_{-\alpha}^{(k)}\big\}_{\alpha \in \hat{\Delta}_+, 1 \leq k \leq \mult \alpha} \cup \big\{{-}x_i^+\big\} \setminus \big\{x_i^-\big\}.
\end{gather*}
Hence we obtain
\begin{gather*}
(T_i \otimes T_i)\Omega_+ - \Omega_+ = \big({-}x_i^+\big) \otimes \big({-}x_i^-\big) - x_i^- \otimes x_i^+.\tag*{\qed}
\end{gather*}\renewcommand{\qed}{}
\end{proof}

\begin{proof}[Proof of Proposition~\ref{prop:compatibility_coproduct}]
We use
\begin{gather*}
\big[\square\big(x_i^+\big), \Omega_+\big] = -x_i^+ \otimes h_i, \qquad \big[\square\big(x_i^-\big), \Omega_+\big] = h_i \otimes x_i^-
\end{gather*}
(see \cite[Lemma~4.2]{MR3861718}).

For $X=x_j^{\pm}, h_j$, both $\Delta T_i(X)$ and $(T_i \otimes T_i) \Delta(X)$ are equal to $\square\, T_i(X)$. Hence it is enough to show $\Delta T_i\big(x_{j,1}^+\big) = (T_i \otimes T_i)\Delta\big(x_{j,1}^+\big)$ since $\affY$ is generated by $x_j^{\pm}$, $h_j$, $x_{j,1}^+$ ($j \in \bbZ/N\bbZ$).

First assume $i=j$. We claim that the both sides coincide with
\begin{gather*}
\square \, T_i\big(x_{i,1}^+\big) +\hbar \big(\big[1 \otimes x_i^-, \Omega_+\big] + x_i^- \otimes h_i \big).
\end{gather*}
The left-hand side is
\begin{gather*}
\Delta\big({-}x_{i,1}^- + \tfrac{\hbar}{2}\big\{h_i,x_i^-\big\}\big) \\
\qquad{} = -\big(\square\big(x_{i,1}^-\big) + \hbar\big[x_i^- \otimes 1,\Omega_+\big]\big) + \tfrac{\hbar}{2} \square\big(\big\{h_i,x_i^-\big\}\big) + \hbar\big(h_i \otimes x_i^- + x_i^- \otimes h_i\big) \\
\qquad{} = \square\,T_i\big(x_{i,1}^+\big) + \hbar \big({-}\big[x_i^- \otimes 1,\Omega_+\big] + h_i \otimes x_i^- + x_i^- \otimes h_i \big),
\end{gather*}
hence we see the claim by $-\big[x_i^- \otimes 1,\Omega_+\big] + h_i \otimes x_i^- = \big[1 \otimes x_i^-,\Omega_+\big]$. The right-hand side is
\begin{gather*}
(T_i \otimes T_i)\big(\square\big(x_{i,1}^+\big)-\hbar\big[1 \otimes x_i^+,\Omega_+\big]\big) = \square\,T_i\big(x_{i,1}^+\big) - \hbar\big[1 \otimes T_i\big(x_i^+\big), (T_i \otimes T_i)\Omega_+\big]\\
\qquad {} = \square\,T_i\big(x_{i,1}^+\big) + \hbar\big[1 \otimes x_i^-, \Omega_+ + x_i^+ \otimes x_i^- - x_i^- \otimes x_i^+\big]\\
\qquad {} = \square\,T_i\big(x_{i,1}^+\big) + \hbar \big(\big[1 \otimes x_i^-, \Omega_+\big] + x_i^- \otimes h_i\big),
\end{gather*}
hence the claim holds.

Next assume $a_{ij}=0$. Then the left-hand side is $\Delta(x_{j,1}^+)$. The right-hand side is
\begin{gather*}
(T_i \otimes T_i)\big(\square\big(x_{j,1}^+\big)-\hbar\big[1 \otimes x_j^+,\Omega_+\big]\big) = \square\, T_i\big(x_{j,1}^+\big) - \hbar\big[1 \otimes T_i\big(x_j^+\big), (T_i \otimes T_i)\Omega_+\big]\\
\qquad{} = \square\big(x_{j,1}^+\big) - \hbar\big[1 \otimes x_j^+, \Omega_+ + x_i^+ \otimes x_i^- - x_i^- \otimes x_i^+\big]\\
\qquad{} = \square\big(x_{j,1}^+\big) - \hbar\big[1 \otimes x_j^+, \Omega_+\big] = \Delta\big(x_{j,1}^+\big),
\end{gather*}
hence the claim holds.

Finally assume $a_{ij}=-1$. We claim that the both sides coincide with
\begin{gather*}
\square\,T_i\big(x_{j,1}^+\big) -\hbar \big(\big[1 \otimes T_i\big(x_j^+\big), \Omega_+\big] - x_i^+ \otimes x_j^+\big).
\end{gather*}
The left-hand side is
\begin{gather*}
\Delta\big(\big[x_i^+, x_{j,1}^+\big]\big) = \big[\square\big(x_i^+\big), \square\big(x_{j,1}^+\big) - \hbar \big[1 \otimes x_j^+,\Omega_+\big]\big]\\
\hphantom{\Delta\big(\big[x_i^+, x_{j,1}^+\big]\big)}{} = \square\,T_i\big(x_{j,1}^-\big) - \hbar \big[\square\big(x_i^+\big),\big[1 \otimes x_j^+,\Omega_+\big]\big].
\end{gather*}
Then we see the claim by
\begin{gather*}
\big[\square\big(x_i^+\big),\big[1 \otimes x_j^+,\Omega_+\big]\big]=\big[\big[\square\big(x_i^+\big),1 \otimes x_j^+\big],\Omega_+\big] + \big[1 \otimes x_j^+,\big[\square\big(x_i^+\big),\Omega_+\big]\big] \\
\qquad {}= \big[1 \otimes T_i\big(x_j^+\big), \Omega_+\big] + \big[1 \otimes x_j^+, -x_i^+ \otimes h_i\big] = \big[1 \otimes T_i\big(x_j^+\big), \Omega_+\big] - x_i^+ \otimes x_j^+.
\end{gather*}
The right-hand side is
\begin{gather*}
(T_i \otimes T_i)\big(\square\big(x_{j,1}^+\big)-\hbar\big[1 \otimes x_j^+,\Omega_+\big]\big) = \square\,T_i\big(x_{j,1}^+\big) - \hbar\big[1 \otimes T_i\big(x_j^+\big), (T_i \otimes T_i)\Omega_+\big]\\
\qquad {}= \square\,T_i\big(x_{j,1}^+\big) - \hbar\big[1 \otimes T_i\big(x_j^+\big), \Omega_+ + x_i^+ \otimes x_i^- - x_i^- \otimes x_i^+\big]\\
\qquad{} = \square\,T_i\big(x_{j,1}^+\big) - \hbar\big(\big[1 \otimes T_i\big(x_j^+\big), \Omega_+\big] - x_i^+ \otimes x_j^+\big),
\end{gather*}
hence the claim holds.
\end{proof}

\section{Heisenberg generators}\label{section4}

\subsection[Affine Lie algebra $\hat{\gl}_N$]{Affine Lie algebra $\boldsymbol{\hat{\gl}_N}$}
Let $\gl_N$ be the complex general linear Lie algebra consisting of $N \times N$ matrices. We denote by $E_{i,j}$ the matrix unit with $(i,j)$-th entry~$1$, and set $\bfid = \sum\limits_{i=1}^N E_{i,i}$. The indices~$i$,~$j$ of $E_{i,j}$ are regarded as elements of $\bbZ/N\bbZ$. The transpose of an element $X$ of $\gl_N$ is denoted by ${}^t X$.

Let $\hat{\gl}_N = \gl_N \otimes \bbC[t,t^{-1}] \oplus \bbC c$ be the affine Lie algebra whose Lie bracket is given by
\begin{gather*}
\big[X \otimes t^{r}, Y \otimes t^{s}\big] = [X,Y] \otimes t^{r+s} + r \delta_{r+s,0} \tr(XY) c, \qquad \text{$c$ is central}.
\end{gather*}
We denote the element $X \otimes t^s$ by $X(s)$. We identify the generators as usual:
\begin{alignat*}{4}
& x_0^+= E_{N,1}(1), \qquad && x_0^{-} = E_{1,N}(-1), \qquad && h_0= E_{N,N} - E_{1,1} + c,&\\
& x_{i}^{+}= E_{i,i+1}, \qquad && x_i^{-}= E_{i+1,i}, \qquad &&h_i= E_{i,i}-E_{i+1,i+1}, \qquad i \neq 0.&
\end{alignat*}

We define automorphisms analogous to $\omega$ and $\rho$ for $\hat{\gl}_N$. Let $\omega$ be the anti-automorphism of~${U}\big(\hat{\gl}_N\big)$ defined by $\omega(X(s))= {}^{t}X(-s)$ and $\omega(c)=c$.
The assignment
\begin{gather*}
x_i^{\pm} \mapsto x_{i-1}^{\pm}, \qquad h_i \mapsto h_{i-1}, \qquad c \mapsto c, \qquad \bfid(s) \mapsto \bfid(s) + \delta_{s,0}c
\end{gather*}
gives an algebra automorphism $\rho$ of $U\big(\hat{\gl}_N\big)$.

\begin{lem}\label{lem:rho}We have $\rho(E_{ij}(s)) = E_{i-1,j-1}(s+\delta_{i,1}-\delta_{j,1}) + \delta_{s,0}\delta_{i,1}\delta_{j,1}c$.
\end{lem}
\begin{proof}We can show the identities for $i \neq j$ inductively from those for the Chevalley generators. Then we can show
\begin{gather}
\rho(h_i(s)) = \begin{cases}
h_{-\theta}(s) & \text{if $i=1$},\\
h_{i-1}(s) & \text{otherwise}
\end{cases}\label{eq:rho_h}
\end{gather}
for $s \neq 0$. Indeed we have $\rho(h_1(s)) = \rho([E_{1,2}(s), E_{2,1}]) = [E_{N,1}(s+1), E_{1,N}(-1)] = h_{-\theta}(s)$ for $i=1$. The other cases are similarly proved. The identity (\ref{eq:rho_h}) will be used later.

Let us consider the case $i = j$. The case $s=0$ is proved as follows. Note that the identity $\bfid = \sum\limits_{i=1}^{N-1} i h_i + N E_{N,N}$ holds. Applying $\rho$ to the both sides, we obtain
\begin{gather*}
\bfid + c = \sum_{i=1}^{N-1} i h_{i-1} + N \rho(E_{N,N}).
\end{gather*}
The right-hand side is equal to
\begin{gather*}
\sum_{i=1}^{N-1} i h_i - N h_{N-1} + c + N \rho(E_{N,N}),
\end{gather*}
hence we have $\rho(E_{N,N}) = E_{N,N} + h_{N-1} = E_{N-1,N-1}$. Then we can inductively show
\begin{gather*}
\rho(E_{i,i}) = \rho(h_i + E_{i+1,i+1}) = h_{i-1} + E_{i,i} = E_{i-1,i-1}
\end{gather*}
for $i=N-1,N-2,\ldots,2$. For $i=1$, we have $\rho(E_{1,1}) = \rho(h_1 + E_{2,2}) = h_{0} + E_{1,1} = E_{N,N} + c$. The case $s \neq 0$ is similarly proved by considering
\begin{gather*}
\bfid(s) = \sum_{i=1}^{N-1} i h_i(s) + N E_{N,N}(s)
\end{gather*}
and (\ref{eq:rho_h}).
\end{proof}

We similarly define an algebra automorphism
\begin{gather*}
T_i = \exp\ad x_i^+ \exp\ad\big({-}x_i^-\big) \exp\ad x_i^+
\end{gather*}
of $U\big(\hat{\gl}_N\big)$.

\begin{lem}\label{lem:higher}We have $T_i(X(s))=T_i(X)(s)$ for $X \in \gl_N$ if $i \neq 0$.
\end{lem}
\begin{proof}Obvious from the definition of $T_i$ and the Lie bracket of $\hat{\gl}_N$.
\end{proof}

\begin{lem}\label{lem:diagonal}We have
\begin{gather*}
T_{0}(E_{j,j}(s)) = \begin{cases}
E_{1,1}(s)-\delta_{s,0}c &\text{if } j=N, \\
E_{N,N}(s)+\delta_{s,0}c &\text{if } j=1, \\
E_{j,j}(s) &\text{otherwise}
\end{cases}
\end{gather*}
and
\begin{gather*}
T_{i}(E_{j,j}(s)) = \begin{cases}
E_{i+1,i+1}(s) &\text{if } j=i, \\
E_{i,i}(s) &\text{if } j=i+1,\\
E_{j,j}(s) \ &\text{otherwise},
\end{cases}\qquad i \neq 0.
\end{gather*}
\end{lem}
\begin{proof}We show the case $i \neq 0$. By Lemma~\ref{lem:higher}, it is enough to consider the case $s=0$. Apply~$T_i$ to the identity $\bfid = \sum\limits_{j=1}^{N-1} j h_j + N E_{N,N}$. Then the left-hand side is $T_i(\bfid) = \bfid$ and the right-hand side is
\begin{gather*}
T_i \left(\sum_{j=1}^{N-1} j h_j + N E_{N,N}\right) = \begin{cases}
\displaystyle\sum_{j=1}^{N-1} j h_j - N h_{N-1} + N T_{N-1}(E_{N,N}) & \text{if $i = N-1$}, \vspace{1mm}\\
\displaystyle\sum_{j=1}^{N-1} j h_j + N T_i(E_{N,N}) & \text{otherwise}.
\end{cases}
\end{gather*}
This shows
\begin{gather*}
T_i (E_{N,N}) = \begin{cases}
E_{N-1,N-1} & \text{if $i = N-1$}, \\
E_{N,N} & \text{otherwise}.
\end{cases}
\end{gather*}
Now let $i=1$. We can inductively show that $E_{N-1,N-1}, E_{N-2,N-2}, \ldots, E_{3,3}$ are invariant under~$T_1$ and
\begin{gather*}
T_1(E_{2,2}) = T_1(h_2 + E_{3,3}) = h_2 + h_1 + E_{3,3} = E_{1,1},\\
T_1(E_{1,1}) = T_1(h_1 + E_{2,2}) = -h_1 + E_{1,1} = E_{2,2}.
\end{gather*}
Thus we have shown the assertion for $i=1$. Similarly we can show the other cases.

The case $i=0$ is obtained from the case $i=1$ by applying $\rho$ and using Lemma~\ref{lem:rho}.
\end{proof}

\begin{lem}\label{lem:0_to_k}We have $T_0 T_1 \cdots T_{k-1} (x_k^+) = \begin{cases}
E_{N,k+1}(1) & \text{if $0 \leq k \leq N-2$},\\
-E_{N,1}(2) & \text{if $k=N-1$}.
\end{cases}$
\end{lem}
\begin{proof} First we prove the assertion for $0 \leq k \leq N-2$ by induction on $k$. The case $k=0$ is trivial. Assume that it holds for~$k$, then
\begin{gather*}
T_0 T_1 \cdots T_{k-1} T_{k} \big(x_{k+1}^+\big) = T_0 T_1 \cdots T_{k-1} \big(\big[x_{k}^+, x_{k+1}^+\big]\big) = \big[T_0 T_1 \cdots T_{k-1}\big(x_{k}^+\big), x_{k+1}^+\big] \\
\hphantom{T_0 T_1 \cdots T_{k-1} T_{k} \big(x_{k+1}^+\big)}{} = [E_{N,k+1}(1), E_{k+1,k+2}] = E_{N,k+2}(1).
\end{gather*}
Next we prove the assertion for $k=N-1$. We have
\begin{gather*}
T_0 T_1 \cdots T_{N-2} \big(x_{N-1}^+\big) = T_0 T_1 \cdots T_{N-3} \big(\big[x_{N-2}^+, x_{N-1}^+\big]\big) = \big[T_0 T_1 \cdots T_{N-3}\big(x_{N-2}^+\big), T_0\big(x_{N-1}^+\big)\big] \\
\hphantom{T_0 T_1 \cdots T_{N-2} \big(x_{N-1}^+\big)}{} = [E_{N,N-1}(1), -E_{N-1,1}(1)] = -E_{N,1}(2).
\end{gather*}
The proof is complete.
\end{proof}

\begin{lem}Let $i \leq j$. We have
\begin{gather}
T_{i} T_{i+1} \cdots T_{j-1} \big(x_{j}^+\big) = E_{i,j+1},\qquad T_{i} T_{i+1} \cdots T_{j-1} \big(x_{j}^-\big) = E_{j+1,i}, \label{eq:increase}\\
T_{j} T_{j-1} \cdots T_{i+1} \big(x_{i}^+\big) = (-1)^{j-i}E_{i,j+1},\qquad T_{j} T_{j-1} \cdots T_{i+1} \big(x_{i}^-\big) = (-1)^{j-i}E_{j+1,i}. \label{eq:decrease}
\end{gather}
\end{lem}
\begin{proof}The assertion is easily proved by induction.
\end{proof}

\begin{lem}We have
\begin{alignat}{3}
&T_i(E_{i+1,j})= E_{i,j} \qquad && \text{if $j+1 \leq i \leq N-1$ \ or \ $i+2 \leq j$}, & \label{eq:TE-1}\\
&T_i(E_{j,i+1})= E_{j,i} \qquad &&\text{if $1 \leq i \leq j-2$ \ or \ $j \leq i-1$}, & \label{eq:TE-2}\\
&T_i(E_{i,1})= -E_{i+1,1} \qquad &&\text{if $2 \leq i \leq N-1$}, & \label{eq:TE-3}\\
&T_i(E_{2,i})= -E_{2,i+1} \qquad &&\text{if $3 \leq i \leq N-1$}. & \label{eq:TE-4}
\end{alignat}
\end{lem}
\begin{proof}The identity (\ref{eq:TE-2}) is deduced from (\ref{eq:TE-1}) by applying $\omega$. We use (\ref{eq:increase}) to show the other identities as
\begin{gather*}
T_i(E_{i+1,j}) = \begin{cases}
T_i T_j T_{j+1} {\cdots} T_{i-1}\big(x_i^-\big) = T_j T_{j+1} {\cdots} T_{i-2} \big(x_{i-1}^-\big) = E_{i,j} &\text{if $j+1 \leq i \leq N-1$}, \\
T_i T_{i+1} T_{i+2} \cdots T_{j-1}\big(x_j^+\big) = E_{i,j} & \text{if $i+2 \leq j$},
\end{cases}
\\
T_i(E_{i,1}) = T_i T_{1} T_{2} \cdots T_{i-2}\big(x_{i-1}^-\big) = T_{1} T_{2} \cdots T_{i-2} T_i \big(x_{i-1}^-\big)= T_{1} T_{2} \cdots T_{i-2} \big(\big[x_{i-1}^-,x_i^-\big]\big)\\
\hphantom{T_i(E_{i,1})}{} = \big[T_{1} T_{2} \cdots T_{i-2} \big(x_{i-1}^-\big),x_i^-\big] = [E_{i,1}, E_{i+1,i}] = -E_{i+1,1},\\
T_i(E_{2,i}) = T_i T_{2} T_{3} \cdots T_{i-2}\big(x_{i-1}^+\big) = T_{2} T_{3} \cdots T_{i-2} T_i \big(x_{i-1}^+\big) = T_{2} T_{3} \cdots T_{i-2} \big(\big[x_{i}^+,x_{i-1}^+\big]\big)\\
\hphantom{T_i(E_{2,i})}{} = \big[x_i^+, T_{2} T_{3} \cdots T_{i-2} \big(x_{i-1}^+\big)\big] = [E_{i,i+1}, E_{2,i}] = -E_{2,i+1}.\tag*{\qed}
\end{gather*}\renewcommand{\qed}{}
\end{proof}

\subsection{Evaluation map}

The evaluation map for the affine Yangian $\affY$ was introduced by Guay~\cite{MR2323534}. It is an affine analog of the well-known evaluation map from $Y(\fraksl_N)$ to $U(\gl_N)$. Let $U\big(\hat{\gl}_N\big)_{{\rm comp},-}$ be the completion of ${U}\big(\hat{\gl}_N\big)$ defined in \cite[Definition~2.5]{kodera_evaluation}. From now on, we regard the central element~$c$ of~$\hat{\gl}_N$ as a complex number.

The following result can be deduced from a formula for $\ev(H_{i,1})$ $(i \neq 0)$ where $H_{i,1}=h_{i,1} + (i/2)(\ve_1-\ve_2)h_i$, given in \cite[Section~6, pp.~462--463]{MR2323534}. See \cite{kodera_evaluation} for details.
\begin{thm}[\cite{MR2323534, kodera_evaluation}]Assume $\hbar c = N \ve_2$. Then there exists an algebra homomorphism $\ev \colon \affY \to U\big(\hat{\gl}_N\big)_{{\rm comp},-}$ uniquely determined by
\begin{gather*}
\ev(x_{i,0}^{+}) = x_{i}^{+}, \qquad \ev(x_{i,0}^{-}) = x_{i}^{-},\qquad \ev(h_{i,0}) = h_{i},\\
\ev(x_{i,1}^{+}) = \begin{cases}
(1+ N \varepsilon_2) x_{0}^{+} + \hbar \displaystyle\sum_{s \geq 0} \sum_{k=1}^N E_{k,1}(s+1) E_{N,k}(-s) \quad \text{if $i = 0$},\\
(1+ i \varepsilon_2) x_{i}^{+} + \hbar \displaystyle\sum_{s \geq 0} \bigg( \sum_{k=1}^i E_{k,i+1}(s) E_{i,k}(-s)\\
\qquad{} + \displaystyle \sum_{k=i+1}^N E_{k,i+1}(s+1) E_{i,k}(-s-1) \bigg) \quad \text{if $i \neq 0$},
\end{cases}\\
\ev(x_{i,1}^-) = \begin{cases}
(1+ N \varepsilon_2) x_{0}^{-} + \hbar \displaystyle\sum_{s \geq 0} \sum_{k=1}^N E_{k,N}(s) E_{1,k}(-s-1)\quad \text{if $i = 0$},\\
(1+ i \varepsilon_2) x_{i}^{-} + \hbar \displaystyle\sum_{s \geq 0} \bigg( \sum_{k=1}^i E_{k,i}(s) E_{i+1,k}(-s) \\
\qquad {} + \displaystyle\sum_{k=i+1}^N E_{k,i}(s+1) E_{i+1,k}(-s-1) \bigg) \quad \text{if $i \neq 0$},
\end{cases}\\
\ev(h_{i,1}) = \begin{cases}
(1+ N \varepsilon_2) h_{0} - \hbar E_{N,N} (E_{1,1}-c) \\
\qquad {}+ \hbar \displaystyle\sum_{s \geq 0} \sum_{k=1}^{N} ( E_{k,N}(s) E_{N,k}(-s) - E_{k,1}(s+1) E_{1,k}(-s-1) ) \quad \text{if $i = 0$},\\
(1+ i \varepsilon_2) h_{i} - \hbar E_{i,i}E_{i+1,i+1}+ \hbar \displaystyle\sum_{s \geq 0} \bigg( \sum_{k=1}^{i} E_{k,i}(s) E_{i,k}(-s) \\
\qquad{}+ \displaystyle\sum_{k=i+1}^{N} E_{k,i}(s+1) E_{i,k}(-s-1) - \displaystyle\sum_{k=1}^{i}E_{k,i+1}(s) E_{i+1,k}(-s) \\
\qquad {} - \displaystyle\sum_{k=i+1}^{N} E_{k,i+1}(s+1) E_{i+1,k}(-s-1) \bigg)\quad \text{if $i \neq 0$}.
\end{cases}
\end{gather*}
\end{thm}
We will use the following property.
\begin{prop}\label{prop:ev_and_auto}We have
\begin{gather*}
\mathrm{(i)}\ \omega \circ \ev = \ev \circ \,\omega, \qquad \mathrm{(ii)}\ \rho \circ \ev = \ev \circ\, \rho,\qquad \mathrm{(iii)}\ T_i \circ \ev = \ev \circ\, T_i.
\end{gather*}
\end{prop}
\begin{proof}The assertion (i) is immediate from the definition of $\ev$. The assertion (ii) is stated in \cite[Section~6, p.~463]{MR2323534} and a proof is given in \cite[Proposition~3.6]{kodera_evaluation}. Since~$\ev$ is the identity on the subalgebra $U\big(\affsl\big)$ and $T_i$ is defined via the generators of $U\big(\affsl\big)$, the assertion (iii) holds.
\end{proof}

\subsection{Construction of Heisenberg generators}

We construct elements $a_m$ ($m \in \bbZ$) of the affine Yangian $\affY$ such that $\ev(a_m) = \bfid(m)$ under the assumption $\ve_2 \neq 0$.

First consider the case $m=0$.
\begin{prop}\label{prop:zero-mode} We have
\begin{gather*} \ev \left( \sum_{i=0}^{N-1} \tilde{h}_{i,1} \right)= \ve_2 \bfid + c + \tfrac{\hbar}{2} c^2.
\end{gather*}
\end{prop}
\begin{proof}
Put
\begin{gather*}
A_0 = \sum_{s \geq 0} \sum_{k=1}^N E_{k,N}(s)E_{N,k}(-s),\qquad
B_0 = \sum_{s \geq 0} \sum_{k=1}^N E_{k,1}(s+1)E_{1,k}(-s-1),\\
A_i = \sum_{s \geq 0} \left( \sum_{k=1}^i E_{k,i}(s)E_{i,k}(-s) + \sum_{k=i+1}^N E_{k,i}(s+1)E_{i,k}(-s-1) \right), \qquad i \neq 0,\\
B_i = \sum_{s \geq 0} \left( \sum_{k=1}^i E_{k,i+1}(s)E_{i+1,k}(-s) + \sum_{k=i+1}^N E_{k,i+1}(s+1)E_{i+1,k}(-s-1) \right), \qquad i \neq 0,
\end{gather*}
so that
\begin{gather*}
\ev\big(\tilde{h}_{i,1}\big) = \begin{cases}
(1+ N \varepsilon_2) h_{0} - \tfrac{\hbar}{2} \big(E_{N,N}^2 + (E_{1,1}-c)^2\big) + \hbar(A_0-B_0)& \text{if $i=0$},\\
(1+ i \varepsilon_2) h_{i} - \tfrac{\hbar}{2} \big(E_{i,i}^2 + E_{i+1,i+1}^2\big) + \hbar(A_i-B_i)& \text{otherwise}.
\end{cases}
\end{gather*}
Here we use
\begin{gather*}
-\hbar E_{N,N}(E_{1,1}-c) - \tfrac{\hbar}{2} h_0^2 = - \tfrac{\hbar}{2} \big(E_{N,N}^2 + (E_{1,1}-c)^2\big),\\
-\hbar E_{i,i} E_{i+1,i+1} - \tfrac{\hbar}{2} h_i^2 = - \tfrac{\hbar}{2} \big(E_{i,i}^2 + E_{i+1,i+1}^2\big), \qquad i \neq 0.
\end{gather*}
Then we have $A_{i+1}-B_{i}=E_{i+1,i+1}^2$ for $i=0,\ldots,N-1$, where we regard $A_N=A_0$.
Therefore
\begin{gather*}
\ev \left( \sum_{i=0}^{N-1} \tilde{h}_{i,1} \right)= c + \ve_2 \left(N h_0 + \sum_{i=1}^{N-1} i h_i\right) -\tfrac{\hbar}{2}\left( 2\sum_{i=1}^N E_{i,i}^2 - 2cE_{1,1} + c^2\right) + \hbar \sum_{i=1}^N E_{i,i}^2 \\
\hphantom{\ev \left( \sum_{i=0}^{N-1} \tilde{h}_{i,1} \right)}{} = c + \ve_2 \sum_{i=1}^{N-1} i h_i + N \ve_2 h_0 + \hbar c E_{1,1} - \tfrac{\hbar}{2}c^2.
\end{gather*}
By the assumption $\hbar c = N\ve_2$, we have
\begin{gather*}
N \ve_2 h_0 + \hbar c E_{1,1} - \tfrac{\hbar}{2}c^2 = N \ve_2 (E_{N,N} + c) - \tfrac{\hbar}{2}c^2 = N \ve_2 E_{N,N} + \tfrac{\hbar}{2}c^2.
\end{gather*}
The assertion holds since we have $\bfid = \sum\limits_{i=1}^{N-1} i h_i + NE_{N,N}$.
\end{proof}

Assume $\hbar c = N \ve_2$ and $\ve_2 \neq 0$. Put
\begin{gather*}
a_0 = \dfrac{1}{\ve_2} \left( \sum_{i=0}^{N-1} \tilde{h}_{i,1} - c - \tfrac{\hbar}{2} c^2\right).
\end{gather*}
Then we have $\ev(a_0) = \bfid$ by Proposition~\ref{prop:zero-mode}.

Next consider the case $m \geq 1$. For each $i \in \bbZ/N\bbZ$ and a fixed $m$, define an element $w(i,m)$ of the affine Weyl group $\hat{W}$ by
\begin{gather*}
w(i,m) = t_{-\alpha_{i+1}}^{m-1} s_{i+1} s_{i+2} \cdots s_{i-3} s_{i-2}.
\end{gather*}
This element has the property $w(i,m)(\alpha_{i-1}) = -\alpha_{i} + m \delta$. Hence the elements
\begin{gather*}
\big[x_{i}^{+}, T_{w(i,m)}\big(x_{i-1}^{+}\big)\big],\qquad \big[x_{i}^{+}, T_{w(i,m)}\big(x_{i-1,1}^{+}\big)\big]
\end{gather*} have weight $m\delta$.
We will see
\begin{gather*}
\big[x_{i}^+, T_{w(i,m)}\big(x_{i-1}^+\big)\big] = (-1)^{m-1} \times \begin{cases}
h_{-\theta}(m) & \text{if $i=0$},\\
h_i(m) & \text{otherwise}
\end{cases}
\end{gather*}
in Lemma~\ref{lem:braket_pre}, and will compute the value of
\begin{gather*}
\ev \big( \big[x_{i}^{+}, T_{w(i,m)}\big(x_{i-1,1}^{+}\big)\big]\big) = \big[x_{i}^{+}, T_{w(i,m)}\ev\big(x_{i-1,1}^{+}\big)\big]
\end{gather*}
in Proposition~\ref{prop:braket}. Then we will take the summation over $i$ in Proposition~\ref{prop:Heisenberg}. The result will yield a construction of the elements $a_m$ in Theorem~\ref{thm:image}.

By Lemma~\ref{lem:rho_omega_and_T}(ii), $\rho \circ T_{w(i,m)} = T_{w(i-1,m)} \circ \rho$ holds. The case $i=1$ will be important. Note that $w(1,1)=s_2 s_3 \cdots s_{N-1}$ and $t_{-\alpha_2} = s_2 s_3 \cdots s_0 s_1 s_0 \cdots s_3$.
\begin{lem}\label{lem:diagonal2}We have
\begin{gather*}
T_{t_{-\alpha_2}}(E_{i,i}(s)) = \begin{cases}
E_{1,1}(s) & \text{if $i=1$}, \\
E_{2,2}(s) + \delta_{s,0}c & \text{if $i=2$}, \\
E_{3,3}(s) - \delta_{s,0}c & \text{if $i=3$}, \\
E_{i,i}(s) & \text{if $4 \leq i \leq N$}
\end{cases}
\end{gather*}
for any $s \in \bbZ$.
\end{lem}
\begin{proof}The assertion is easily proved by Lemma~\ref{lem:diagonal}.
\end{proof}
\begin{lem}\label{lem:i=1_partial}We have
\begin{gather*}
T_{w(1,m)}(E_{1,1}(s)) = E_{1,1}(s),\qquad T_{w(1,m)}(E_{N,N}(-s)) = E_{2,2}(-s) + \delta_{s,0}(m-1)c
\end{gather*}
for any $s \in \bbZ$.
\end{lem}
\begin{proof}The assertion follows from
\begin{gather*}
T_2 T_3 \cdots T_{N-1} (E_{1,1}(s)) = E_{1,1}(s), \qquad T_2 T_3 \cdots T_{N-1} (E_{N,N}(-s)) = E_{2,2}(-s)
\end{gather*}
and Lemma~\ref{lem:diagonal2}.
\end{proof}

\begin{lem}\label{lem:classical_part} We have
\begin{gather*} T_{w(i,m)}(x_{i-1}^+) = (-1)^{m-1} \times \begin{cases}
E_{1,N}(m-1) & \text{if $i=0$},\\
E_{i+1,i}(m) & \text{otherwise}.
\end{cases}
\end{gather*}
\end{lem}
\begin{proof}We prove the assertion for $i=1$. The other cases are deduced from this case by applying~$\rho$ and Lemma~\ref{lem:rho}. First assume $m=1$. Recall $T_{w(1,1)}=T_2 T_3 \cdots T_{N-1}$. We have
\begin{gather*}
T_2 T_3 \cdots T_{N-1}\big(x_0^+\big) = T_2 T_3 \cdots T_{N-2} \big(\big[x_{N-1}^+,x_0^+\big]\big) = \big[T_2 T_3 \cdots T_{N-2} \big(x_{N-1}^+\big),x_0^+\big] \\
\hphantom{T_2 T_3 \cdots T_{N-1}\big(x_0^+\big)}{} \overset{\text{by (\ref{eq:increase})}}{=} [E_{2,N}, E_{N,1}(1)] = E_{2,1}(1),
\end{gather*}
hence the case $m=1$ is proved. Next consider the case $m \geq 2$. Since the case $m=2$ yields the equality for general $m \geq 2$ inductively, it is enough to prove
\begin{gather*}
T_{t_{-\alpha_2}}(E_{2,1}(1))=-E_{2,1}(2).
\end{gather*}
Since $E_{2,1}(1)$ is invariant under $T_i$ for $3 \leq i \leq N-1$, we have
\begin{gather}
T_{t_{-\alpha_2}}(E_{2,1}(1)) = T_2 T_3 \cdots T_0 T_1 T_0 (E_{2,1}(1)).\label{eq:E21}
\end{gather}
We have
\begin{gather*}
T_0 T_1 T_0 (E_{2,1}(1)) = T_0 T_1 T_0 \big(T_2 T_3 \cdots T_{N-1}\big(x_0^+\big)\big) = T_0 T_1 T_2 T_3 \cdots T_{N-2}\big(x_{N-1}^+\big) = -E_{N,1}(2).
\end{gather*}
Here the second equality follows from the braid relations and Proposition~\ref{prop:iji_pre}, and the last from Lemma~\ref{lem:0_to_k}. Then the right-hand side of (\ref{eq:E21}) is
\begin{gather*}
T_2 T_3 \cdots T_{N-1} (-E_{N,1}(2)) \overset{\text{by Lemma~\ref{lem:higher}}}{=} - (T_2 T_3 \cdots T_{N-1} (E_{N,1}) )(2) \overset{\text{by (\ref{eq:TE-1})}}{=} -E_{2,1}(2),
\end{gather*}
hence the assertion is proved.
\end{proof}

\begin{lem}\label{lem:braket_pre}
We have
\begin{gather*} \big[x_{i}^+, T_{w(i,m)}\big(x_{i-1}^+\big)\big] = (-1)^{m-1} \times \begin{cases}
h_{-\theta}(m) & \text{if $i=0$},\\
h_i(m) & \text{otherwise}.
\end{cases}
\end{gather*}
\end{lem}

\begin{proof}
Immediate from Lemma~\ref{lem:classical_part}.
\end{proof}

\begin{lem}\label{lem:i=1}We have
\begin{gather*}
T_{w(1,m)}(E_{k,1}(s+1)E_{N,k}(-s)) \\
\qquad{} = (-1)^{m-1} \times\begin{cases}
E_{1,1}(s+1)E_{2,1}(-s+m-1) & \text{if $k=1$},\\
E_{3,1}(s-m+2)E_{2,3}(-s+2m-2) & \text{if $k=2$},\\
E_{k+1,1}(s+1)E_{2,k+1}(-s+m-1) & \text{if $3 \leq k \leq N-1$},\\
E_{2,1}(s+m)\Big( E_{2,2}(-s) + \delta_{s,0}(m-1)c \Big) & \text{if $k=N$}.
\end{cases}
\end{gather*}
\end{lem}

\begin{proof}We prove
\begin{gather}
T_{w(1,m)}(E_{k,1}(s+1)) = \begin{cases}
E_{1,1}(s+1) & \text{if $k=1$},\\
(-1)^{(m-1)(N+1)} (-E_{3,1}(s-m+2)) & \text{if $k=2$},\\
-E_{k+1,1}(s+1) & \text{if $3 \leq k \leq N-1$},\\
(-1)^{m-1} E_{2,1}(s+m) & \text{if $k=N$}
\end{cases}\label{eq:i=1former}
\end{gather}
and
\begin{gather}
T_{w(1,m)}(E_{N,k}(-s)) = \begin{cases}
(-1)^{m-1} E_{2,1}(-s+m-1) & \text{if $k=1$},\\
(-1)^{(m-1)N} (- E_{2,3}(-s+2m-2)) & \text{if $k=2$},\\
(-1)^{m-1}(- E_{2,k+1}(-s+m-1)) & \text{if $3 \leq k \leq N-1$},\\
E_{2,2}(-s) + \delta_{s,0}(m-1)c & \text{if $k=N$}.
\end{cases}\label{eq:i=1latter}
\end{gather}
The equalities (\ref{eq:i=1former}) for $k=1$ and (\ref{eq:i=1latter}) for $k=N$ follow from Lemma~\ref{lem:i=1_partial}.

Consider (\ref{eq:i=1former}) for $k=N$ and (\ref{eq:i=1latter}) for $k=1$. Note that Lemma~\ref{lem:classical_part} for $i=1$ is nothing but (\ref{eq:i=1former}) for $k=N$ and $s=0$. We can prove the other cases by applying $[-, E_{1,1}(\pm s)]$ to this case as
\begin{gather*}
\big[T_{w(1,m)}(E_{N,1}(1)), E_{1,1}(\pm s)\big] = T_{w(1,m)}\big(\big[E_{N,1}(1), E_{1,1}(\pm s)\big]\big) = T_{w(1,m)}\big(E_{N,1}(1 \pm s)\big).
\end{gather*}
Here we use the fact that $E_{1,1}(\pm s)$ is invariant under $T_{w(1,m)}$ proved in Lemma~\ref{lem:i=1_partial}. In the sequel, we prove (\ref{eq:i=1former}) and (\ref{eq:i=1latter}) for $2 \leq k \leq N-1$.

We prove (\ref{eq:i=1former}). First we consider the case $m=1$. We prove
\begin{gather*}
T_{w(1,1)} \big(E_{k,1}(s+1)\big)=-E_{k+1,1}(s+1)
\end{gather*}
for $2 \leq k \leq N-1$ by induction on $k$. Since $T_{w(1,1)}=T_2 T_3 \cdots T_{N-1}$ does not involve $T_0$, it is enough to prove $T_{w(1,1)}(E_{k,1})=-E_{k+1,1}$ by Lemma~\ref{lem:higher}. When $k=2$, we have
\begin{gather*}
T_2 T_3 \cdots T_{N-1}\big(x_1^-\big) = T_2\big(x_1^-\big) = -E_{3,1}.
\end{gather*}
Assume that the assertion holds for $k$. Since we have $E_{k+1,1}=-T_k(E_{k,1})$ by (\ref{eq:TE-3}),
\begin{gather*}
T_2 T_3 \cdots T_{N-1}(E_{k+1,1}) = -T_2 T_3 \cdots T_{N-1} T_k(E_{k,1}) = -T_{k+1} T_2 T_3 \cdots T_{N-1} (E_{k,1})\\
\hphantom{T_2 T_3 \cdots T_{N-1}(E_{k+1,1})}{} = T_{k+1}(E_{k+1,1}) \overset{\text{by (\ref{eq:TE-3})}}{=} -E_{k+2,1}.
\end{gather*}
Thus we have proved the case $m=1$. Next we consider the case $m \geq 2$. Since the case $m=2$ yields the equality for general $m \geq 2$ inductively, it is enough to prove:
\begin{gather}
T_{t_{-\alpha_2}}\big(E_{3,1}(s+1)\big) = (-1)^{N+1}E_{3,1}(s), \label{eq:m=2_k=2}\\
T_{t_{-\alpha_2}}\big(E_{k+1,1}(s+1)\big) = E_{k+1,1}(s+1) \qquad \text{if $3 \leq k \leq N-1$}. \label{eq:m=2_k+1}
\end{gather}
If we prove the assertion for $s=0$, we can prove the other cases by applying $[-,E_{1,1}(s)]$ to this case by using the fact that $E_{1,1}(s)$ is invariant under $T_{t_{-\alpha_2}}$ proved in Lemma~\ref{lem:diagonal2}.
We prove~(\ref{eq:m=2_k=2}) for $s=0$. Since we have $E_{3,1}=-T_2\big(x_1^-\big)$, the left-hand side of~(\ref{eq:m=2_k=2}) for $s=0$ is
\begin{gather*}
-T_2 T_3 \cdots T_0 T_1 T_0 \cdots T_3 \big(T_2 \big(x_1^-\big)(1)\big) \overset{\text{by (\ref{eq:decrease})}}{=} -T_2 T_3 \cdots T_0 T_1 T_0\big( (-1)^{N-2} E_{N,1}(1) \big) \\
\qquad {} = (-1)^{N-1} T_2 T_3 \cdots T_0 T_1 T_0 \big(x_0^+\big) = (-1)^N T_2 T_3 \cdots T_{N-1} \big(x_1^-\big) \\
\qquad{} = (-1)^N T_2\big(x_1^-\big) = (-1)^{N+1} E_{3,1}.
\end{gather*}
We prove (\ref{eq:m=2_k+1}) for $s=0$ by backward induction on $k$. The case $k=N-1$ is proved as
\begin{gather*}
T_2 T_3 \cdots T_0 T_1 T_0 \cdots T_3\big(x_0^+\big) = T_2 T_3 \cdots T_0 T_1 \big(x_{N-1}^+\big) = T_2 T_3 \cdots T_{N-2} \big(x_0^+\big) = x_0^+.
\end{gather*}
Assume that the assertion holds for $k$. Since we have $E_{k,1} = T_k(E_{k+1,1})$ by~(\ref{eq:TE-1}),
\begin{gather*}
T_{t_{-\alpha_2}}\big(E_{k,1}(1)\big) = T_{t_{-\alpha_2}} T_{k}\big(E_{k+1,1}(1)\big) = T_{k} T_{t_{-\alpha_2}}\big(E_{k+1,1}(1)\big) = T_k \big(E_{k+1,1}(1)\big) = E_{k,1}(1).
\end{gather*}
Here the second equality holds by $k \geq 3$. Thus we have proved (\ref{eq:i=1former}).

We prove (\ref{eq:i=1latter}). First we consider the case $m=1$. We prove
\begin{gather*}
T_{w(1,1)} \big(E_{N,k}(-s)\big)=-E_{N,k+1}(-s)
\end{gather*}
for $2 \leq k \leq N-1$ by backward induction on $k$. By Lemma~\ref{lem:higher}, it is enough to prove $T_{w(1,1)}(E_{N,k})=-E_{2,k+1}$. When $k=N-1$, we have
\begin{gather*}
T_2 T_3 \cdots T_{N-1}\big(x_{N-1}^-\big) = -T_2 T_3 \cdots T_{N-2}\big(x_{N-1}^+\big) = -E_{2,N}
\end{gather*}
by (\ref{eq:increase}). Assume that the assertion holds for $k$. Since we have $E_{N,k-1}=T_{k-1}(E_{N,k})$ by (\ref{eq:TE-2}),
\begin{gather*}
T_2 T_3 \cdots T_{N-1}\big(E_{N,k-1}\big) = T_2 T_3 \cdots T_{N-1} T_{k-1}\big(E_{N,k}\big) = T_k T_2 T_3 \cdots T_{N-1} \big(E_{N,k}\big)\\
\hphantom{T_2 T_3 \cdots T_{N-1}\big(E_{N,k-1}\big)}{} = -T_k(E_{2,k+1}) \overset{\text{by (\ref{eq:TE-2})}}{=} -E_{2,k}.
\end{gather*}
Thus we have proved the case $m=1$. Next we consider the case $m \geq 2$. Since the case $m=2$ yields the equality for general $m \geq 2$ inductively, it is enough to prove:
\begin{gather}
T_{t_{-\alpha_2}}\big(E_{2,3}(-s)\big) = (-1)^N E_{2,3}(-s+2), \label{eq:m=2_k=2_second}\\
T_{t_{-\alpha_2}}\big(E_{2,k+1}(-s)\big) = -E_{2,k+1}(-s+1) \qquad \text{if $3 \leq k \leq N-1$}. \label{eq:m=2_k+1_second}
\end{gather}
If we prove the assertion for $s=0$, we can prove the other cases by applying $[E_{2,2}(-s),-]$ to this case by using the fact that $E_{2,2}(-s)$ is invariant under $T_{t_{-\alpha_2}}$ for $s \geq 1$ proved in Lemma~\ref{lem:diagonal2}. We prove (\ref{eq:m=2_k=2_second}) for $s=0$. Since we have
\begin{gather*}
T_{N-1} T_{N-2} \cdots T_3 \big(x_2^+\big) = (-1)^{N-3} E_{2,N}
\end{gather*}
by (\ref{eq:decrease}), and
\begin{gather*}
T_0 T_1 T_0 \big(E_{2,N}\big) = T_1 T_0 T_1 \big(E_{2,N}\big) \overset{\text{by (\ref{eq:TE-1})}}{=} T_1 T_0 \big(E_{1,N}\big)
\overset{\text{by Lemma~\ref{lem:0_to_k}}}{=} T_1\big({-}E_{N,1}(2)\big),
\end{gather*}
the left-hand side of (\ref{eq:m=2_k=2_second}) for $s=0$ is
\begin{gather*}
(-1)^{N-2} T_2 T_3 \cdots T_{N-1} T_1 \big(E_{N,1}(2)\big) = (-1)^{N-2} T_2 T_1 T_3 \cdots T_{N-1} \big(E_{N,1}(2)\big) \\
\qquad{} \overset{\text{by (\ref{eq:TE-1})}}{=} (-1)^{N-2} T_2 T_1 \big(E_{3,1}\big) (2)= (-1)^{N-2} T_2 T_1 \big({-}T_2\big(x_1^-\big)\big)(2) \\
\qquad {} = (-1)^{N-1} T_2 \big(x_2^-\big)(2) = (-1)^N E_{2,3}(2).
\end{gather*}
We prove (\ref{eq:m=2_k+1_second}) for $s=0$ by induction on $k$. Assume $k=3$. Since we have $E_{2,4}=T_2\big(x_3^+\big)$, the left-hand side of (\ref{eq:m=2_k+1_second}) for $s=0$ is
\begin{gather*}
T_2 T_3 \cdots T_0 T_1 T_0 \cdots T_3\big(T_2\big(x_3^+\big)\big) = T_2 T_3 \cdots T_0 T_1 T_0 \cdots T_{4} \big(x_{2}^+\big) = T_2 T_3 \cdots T_0 T_1 \big(x_2^+\big) \\
\qquad \overset{\text{by Lemma~\ref{lem:0_to_k}}}{=} T_2 T_3 \cdots T_{N-1} \big(E_{N,3}(1)\big)\overset{\text{by (\ref{eq:TE-1})} }{=} T_2 T_3 (E_{4,3}) (1) = - T_2\big(x_3^+\big)(1) = -E_{2,4}(1).
\end{gather*}
Assume that the assertion holds for $k$.
Since we have $E_{2,k+2} = -T_{k+1}(E_{2,k+1})$ by (\ref{eq:TE-4}), the left-hand side is
\begin{gather*}
T_{t_{-\alpha_2}}\big(E_{2,k+2}\big) = -T_{t_{-\alpha_2}} T_{k+1}\big(E_{2,k+1}\big) = -T_{k+1} T_{t_{-\alpha_2}}\big(E_{2,k+1}\big)\\
\hphantom{T_{t_{-\alpha_2}}\big(E_{2,k+2}\big)}{} = T_{k+1} \big(E_{2,k+1}\big)(1) = -E_{2,k+2}(1).
\end{gather*}
Here the second equality holds by $k \geq 3$.
\end{proof}

For a fixed $m \geq 1$, put
\begin{gather*}
S_{i,j}(p) = E_{i,j}(p)E_{j,i}(m-p).
\end{gather*}

\begin{prop}\label{prop:braket}
We have
\begin{gather}
\big[x_i^+, T_{w(i,m)} \ev\big(x_{i-1,1}^+\big)\big] \nonumber\\
= (-1)^{m-1} \left(A_i + \hbar \sum_{k=1}^N (P_{i,k}-Q_{i,k}) - \hbar \sum_{s=0}^{m-2} E_{i,i}(s+1)E_{i+1,i+1}(-s+m-1) \right),\label{eq:braket}
\end{gather}
where
\begin{gather*}
A_i = \begin{cases}
(1+(N-1)\ve_2)h_{-\theta}(m) + (m-1)\hbar ch_{-\theta}(m)+ \hbar c E_{N,N}(m) & \text{if $i=0$}, \\
(1+N\ve_2)h_1(m) + (m-1)\hbar ch_{1}(m)& \text{if $i=1$}, \\
(1+(i-1)\ve_2)h_{i}(m) + (m-1)\hbar ch_{i}(m) + \hbar c h_{i}(m) & \text{if $2 \leq i \leq N-1$},
\end{cases}
\\
P_{i,k} = \sum_{s \geq 0} S_{k,i}(p(i,k)), \qquad Q_{i,k} = \sum_{s \geq 0} S_{k,i+1}(q(i,k)).
\end{gather*}
The integers $p(i,k)$ and $q(i,k)$ are given by
\begin{gather*}
p(0,k)= \begin{cases}
s+m-1 & \text{if $k=1$},\\
s-m+1 & \text{if $k=2$},\\
s & \text{if $3 \leq k \leq N-1$},\\
s+1 & \text{if $k=N$},
\end{cases}\\
p(i,k)= \begin{cases}
s & \text{if $1 \leq k \leq i-1$},\\
s+1 & \text{if $k=i$},\\
s+m & \text{if $k=i+1$},\\
s-m+2 & \text{if $k=i+2$},\\
s+1 & \text{if $i+3 \leq k \leq N$},
\end{cases}\qquad 1 \leq i \leq N-2,\\
p(N-1,k)= \begin{cases}
s-m+1 & \text{if $k=1$},\\
s & \text{if $2 \leq k \leq N-2$},\\
s+1 & \text{if $k=N-1$},\\
s+m & \text{if $k=N$},
\end{cases}\qquad
q(i,k)=\begin{cases}
p(0,k)+1 & \text{if $i=0$},\\
p(i,k) &\text{otherwise}.
\end{cases}
\end{gather*}
\end{prop}

\begin{proof}We prove the assertion for $i=1$. By Lemma~\ref{lem:i=1}, we have
\begin{gather*}
(-1)^{m-1}T_{w(1,m)}\ev\big(x_{0,1}^+\big) \\
 \qquad{} = (-1)^{m-1}(1+N\ve_2)T_{w(1,m)}\big(x_0^+\big) + (-1)^{m-1}\hbar \sum_{s \geq 0} \sum_{k=1}^N T_{w(1,m)}\big(E_{k,1}(s+1)E_{N,k}(-s)\big) \\
\qquad{} =(-1)^{m-1}(1+N\ve_2)T_{w(1,m)}\big(x_0^+\big) + (m-1)\hbar c E_{2,1}(m)\\
\qquad \quad{} + \hbar \sum_{s \geq 0} \bigg( E_{1,1}(s+1)E_{2,1}(-s+m-1) + E_{2,1}(s+m)E_{2,2}(-s) \\
 \qquad \quad{}+ E_{3,1}(s-m+2)E_{2,3}(-s+2m-2) + \sum_{k=4}^N E_{k,1}(s+1)E_{2,k}(-s+m-1) \bigg).
\end{gather*}
Therefore
\begin{gather}
(-1)^{m-1} [x_1^+, T_{w(1,m)} \ev(x_{0,1}^+)] = (1+N\ve_2)h_1(m) + (m-1)\hbar c h_1(m) \nonumber\\
{}+ \hbar \sum_{s \geq 0} \bigg( {-}S_{1,2}(s+1) + E_{1,1}(s+1)h_1(-s+m-1) + h_1(s+m)E_{2,2}(-s) + S_{2,1}(s+m)\nonumber\\
{}-S_{3,2}(s-m+2) + S_{3,1}(s-m+2) + \sum_{k=4}^N \big( {-}S_{k,2}(s+1) + S_{k,1}(s+1) \big) \bigg). \label{eq:braket_i=1}
\end{gather}
Here we use Lemma~\ref{lem:braket_pre}. Since we have
\begin{gather*}
\sum_{s \geq 0} \big( E_{1,1}(s+1)h_1(-s+m-1) + h_1(s+m)E_{2,2}(-s) \big)\\
\qquad {}=\sum_{s \geq 0} \big( S_{1,1}(s+1)-S_{2,2}(s+m) \big) - \sum_{s=0}^{m-2} E_{1,1}(s+1)E_{2,2}(-s+m-1),
\end{gather*}
the right-hand side of (\ref{eq:braket_i=1}) is equal to
\begin{gather*}
(1+N\ve_2)h_1(m) + (m-1)\hbar c h_1(m) \\
\qquad{} + \hbar \sum_{s \geq 0} \bigg( \bigg( S_{1,1}(s+1) + S_{2,1}(s+m) + S_{3,1}(s-m+2) + \sum_{k=4}^N S_{k,1}(s+1) \bigg) \\
\qquad{} - \bigg( S_{1,2}(s+1) + S_{2,2}(s+m) + S_{3,2}(s-m+2) + \sum_{k=4}^N S_{k,2}(s+1) \bigg) \bigg)\\
\qquad{} - \hbar \sum_{s=0}^{m-2} E_{1,1}(s+1)E_{2,2}(-s+m-1).
\end{gather*}
Hence the assertion holds for $i=1$. Then apply $\rho$ to (\ref{eq:braket}) for $i=1$. The left-hand side is
\begin{gather*}
\begin{split}&
 \rho\big(\big[x_1^+, T_{w(1,m)} \ev\big(x_{0,1}^+\big)\big]\big)=\big[x_0^+, T_{w(0,m)} \ev\big(x_{0,1}^++\ve_2 x_0^+\big)\big]\\
& \hphantom{\rho\big(\big[x_1^+, T_{w(1,m)} \ev\big(x_{0,1}^+\big)\big]\big)}{} = \big[x_0^+, T_{w(0,m)} \ev\big(x_{0,1}^+\big)\big] + \ve_2 h_{-\theta}(m).
\end{split}
\end{gather*}
Here we use Proposition~\ref{prop:ev_and_auto}(ii) and Lemma~\ref{lem:braket_pre}. Consider the right-hand side. We can see
\begin{gather*}
\rho(A_1) = A_0 + \ve_2 h_{-\theta}(m) - \hbar c E_{N,N}(m),\\
\rho(P_{1,k}) = \begin{cases}
P_{0,N} + E_{N,N}(m)c & \text{if $k=1$},\\
P_{0,k-1} & \text{if $2 \leq k \leq N$},
\end{cases}\qquad
\rho(Q_{1,k}) = Q_{0,k-1},\\
\rho\big(E_{1,1}(s+1)E_{2,2}(-s+m-1)\big) = E_{N,N}(s+1)E_{1,1}(-s+m-1).
\end{gather*}
Hence the assertion holds for $i=0$. Then apply $\rho$ to (\ref{eq:braket}) for $i=0$. Similarly the left-hand side is
\begin{gather*}
 \rho\big(\big[x_0^+, T_{w(0,m)} \ev\big(x_{N-1,1}^+\big)\big]\big)= \big[x_{N-1}^+, T_{w(N-1,m)} \ev\big(x_{N-2,1}^+\big)\big] + \ve_2 h_{N-1}(m),
\end{gather*}
and for the right-hand side, we can see
\begin{gather*}
\rho(A_0) = A_{N-1} + \ve_2 h_{N-1}(m) + \hbar c E_{N,N}(m),\\
\rho(P_{0,k}) = P_{N-1,k-1}, \qquad \rho(Q_{0,k}) = \begin{cases}
Q_{N-1,N} - E_{N,N}(m)c & \text{if $k=1$},\\
Q_{N-1,k-1} & \text{if $2 \leq k \leq N$},
\end{cases}\\
\rho\big(E_{N,N}(s+1)E_{1,1}(-s+m-1)\big) = E_{N-1,N-1}(s+1)E_{N,N}(-s+m-1).
\end{gather*}
Note that $c$ never appears in the last equality as $-s+m-1$ cannot be $0$ for $0 \leq s \leq m-2$. Hence the assertion holds for $i=N-1$. Continuing this process we prove the assertions for $i=N-2, N-3,\ldots,2$ since we have
\begin{gather*}
\rho(A_{i+1}) = A_{i} + \ve_2 h_{i}(m),\qquad \rho(P_{i+1,k}) = P_{i,k-1}, \qquad \rho(Q_{i+1,k}) =Q_{i,k-1},\\
\rho\big(E_{i+1,i+1}(s+1)E_{i+2,i+2}(-s+m-1)\big) = E_{i,i}(s+1)E_{i+1,i+1}(-s+m-1).\tag*{\qed}
\end{gather*}\renewcommand{\qed}{}
\end{proof}
\begin{prop}\label{prop:Heisenberg}
We have
\begin{gather}
\ev\left( (-1)^{m-1} \sum_{i=0}^{N-1} \big[x_{i}^+, T_{w(i,m)}\big(x_{i-1,1}^+\big)\big] \right) = \ve_2 \bfid(m) + \hbar R_m, \label{eq:Heisenberg}
\end{gather}
where
\begin{gather*}
R_m= (-1)^{m-1}\sum_{i=0}^{N-1} x_{i}^+ T_{w(i,m)}\big(x_{i-1}^+\big) + \sum_{p=1}^{m-1} \sum_{1 \leq i \leq j \leq N-1} h_i(p)h_j(m-p) \\
\hphantom{R_m=}{} - \sum_{p=-m+2}^{m-1} \bigg( E_{1,N}(p-1)E_{N,1}(m-p+1) + \sum_{i=0}^{N-2} E_{i+2,i+1}(p)E_{i+1,i+2}(m-p) \bigg) \\
\hphantom{R_m=}{} + \sum_{p=-m+2}^{0} \bigg( E_{1,N-1}(p-1)E_{N-1,1}(m-p+1) +E_{2,N}(p-1)E_{N,2}(m-p+1) \\
\hphantom{R_m=}{} + \sum_{i=0}^{N-3} E_{i+3,i+1}(p)E_{i+1,i+3}(m-p) \bigg).
\end{gather*}
\end{prop}

\begin{rem}The point of the statement of Proposition~\ref{prop:Heisenberg} is as follows: although each term $\ev\big(\big[x_{i}^+, T_{w(i,m)}\big(x_{i-1,1}^+\big)\big]\big)$ lies in the completion of $U\big(\hat{\gl}_N\big)$, we obtain $R_m$ an element of $U\big(\affsl\big)$ as a remainder term after cancellation.
\end{rem}

\begin{proof} We use the notation in Proposition~\ref{prop:braket}. We have
\begin{gather*}
\sum_{i=0}^{N-1} A_i = \ve_2 \left((N-1)h_{-\theta} + N h_1 + \sum_{i=2}^{N-1}(i-1)h_i\right)(m) + \hbar c \left(E_{N,N}+\sum_{i=2}^{N-1} h_i\right)(m) \\
\hphantom{\sum_{i=0}^{N-1} A_i}{} = \ve_2 \left( (N-1)h_{-\theta} + Nh_1 + \sum_{i=2}^{N-1} (N+i-1) h_i + NE_{N,N} \right)(m) \\
\hphantom{\sum_{i=0}^{N-1} A_i}{} = \ve_2 \left( \sum_{i=1}^{N-1} i h_i + NE_{N,N}\right)(m)=\ve_2 \bfid(m).
\end{gather*}
In the second equality we use the condition $\hbar c = N\ve_2$. We compute $\sum\limits_{k=1}^{N} (P_{i+1,k}-Q_{i,k})$ as follows:
\begin{gather*}
\sum_{k=1}^{N} (P_{0,k}-Q_{N-1,k}) = - \sum_{s=0}^{2m-3}S_{1,N}(s-m+1) + \sum_{s=0}^{m-2}S_{2,N}(s-m+1)\\
\hphantom{\sum_{k=1}^{N} (P_{0,k}-Q_{N-1,k})=}{} +S_{N-1,N}(0)+\sum_{s=0}^{m-2} S_{N,N}(s+1),
\\
\sum_{k=1}^{N} (P_{1,k}-Q_{0,k}) = \sum_{s=0}^{m-2} S_{1,1}(s+1) - \sum_{s=0}^{2m-3}S_{2,1}(s-m+2) \\
\hphantom{\sum_{k=1}^{N} (P_{1,k}-Q_{0,k}) =}{} + \sum_{s=0}^{m-2}S_{3,1}(s-m+2)+S_{N,1}(1),
\\
\sum_{k=1}^{N} (P_{i+1,k}-Q_{i,k}) = S_{i,i+1}(0) + \sum_{s=0}^{m-2} S_{i+1,i+1}(s+1)- \sum_{s=0}^{2m-3}S_{i+2,i+1}(s-m+2) \\
\hphantom{\sum_{k=1}^{N} (P_{i+1,k}-Q_{i,k}) =}{} + \sum_{s=0}^{m-2}S_{i+3,i+1}(s-m+2) \qquad \text{for}\ 1 \leq i \leq N-3,
\\
\sum_{k=1}^{N} (P_{N-1,k}-Q_{N-2,k})= \sum_{s=0}^{m-2}S_{1,N-1}(s-m+1)+S_{N-2,N-1}(0)\\
\hphantom{\sum_{k=1}^{N} (P_{N-1,k}-Q_{N-2,k})=}{} +\sum_{s=0}^{m-2} S_{N-1,N-1}(s+1) - \sum_{s=0}^{2m-3}S_{N,N-1}(s-m+2).
\end{gather*}
Hence the assertion holds by
\begin{gather*}
(-1)^{m-1} x_i^+ T_{w(i,m)}\big(x_{i-1}^+\big) = \begin{cases}
S_{N,1}(1) & \text{if $i=0$,}\\
S_{i,i+1}(0) & \text{otherwise}
\end{cases}
\end{gather*}
and
\begin{gather*}
\sum_{i=0}^{N-1} \big( S_{i,i}(s+1) - E_{i,i}(s+1) E_{i+1,i+1}(-s+m-1) \big) \\
\qquad{}= \left( \sum_{i=1}^{N-1} E_{i,i}(s+1) h_{i}(-s+m-1) \right)+ E_{N,N}(s+1)h_{-\theta}(-s+m-1)\\
\qquad{} = \sum_{i=1}^{N-1} (E_{i,i}-E_{N,N})(s+1) h_{i}(-s+m-1) = \sum_{i=1}^{N-1} \sum_{j \geq i} h_j(s+1) h_i(-s+m-1).\tag*{\qed}
\end{gather*}\renewcommand{\qed}{}
\end{proof}

Applying $\omega$ to (\ref{eq:Heisenberg}), we obtain
\begin{gather*}
\ev\left( (-1)^{m-1} \sum_{i=0}^{N-1} \big[T_{w(i,m)}\big(x_{i-1,1}^-\big), x_{i}^-\big] \right) = \ve_2 \bfid(-m) + \hbar \omega(R_m).
\end{gather*}
Now the following theorem has been proved.
\begin{thm}\label{thm:image}Assume $\hbar c =N\ve_2$ and $\ve_2 \neq 0$. Let $R_m$ $(m \geq 1)$ be the element of $U(\hat{\mathfrak{sl}}_N) \subset \affY$ as in Proposition~{\rm \ref{prop:Heisenberg}}, and define for $m \in \bbZ$,
\begin{gather*}
a_{m} = \dfrac{1}{\ve_2} \times \begin{cases}
\displaystyle \left( (-1)^{m-1} \sum_{i=0}^{N-1} \big[x_{i}^+, T_{w(i,m)}\big(x_{i-1,1}^+\big)\big] - \hbar R_m\right) & \text{if $m>0$},\\
\displaystyle \left( \sum_{i=0}^{N-1} \tilde{h}_{i,1} - c - \tfrac{\hbar}{2} c^2\right) & \text{if $m=0$},\\
\displaystyle \left((-1)^{m-1} \sum_{i=0}^{N-1} \big[T_{w(i,-m)}\big(x_{i-1,1}^-\big),x_{i}^-\big] - \hbar \omega(R_{-m})\right) & \text{if $m<0$}.
\end{cases}
\end{gather*}
Then we have $\ev(a_m) = \bfid(m)$. In particular, the image of the evaluation map $\ev$ contains $U\big(\hat{\gl}_N\big)$.
\end{thm}

\begin{cor}Assume $\hbar c =N\ve_2$ and $\ve_2 \neq 0$. Then the pull-back of an irreducible $\hat{\gl}_N$-module by the evaluation map $\ev$ is irreducible as a module of $\affY$.
\end{cor}

\begin{rem}If $\ve_2 = 0$, we cannot apply our construction. In fact, if we assume $c \neq 0$, the condition $\hbar c = N\ve_2$ and $\ve_2 = 0$ implies $\ve_1=0$. The affine Yangian at $\ve_1 = \ve_2 = 0$ is isomorphic to the universal enveloping algebra of the universal central extension of the Lie algebra $\mathfrak{sl}_N\big[s,t^{\pm 1}\big]$. Moreover the evaluation map becomes the genuine evaluation at $s=0$. In this situation, the image of the evaluation map is $U\big(\affsl\big)$, and hence it does not contain the Heisenberg algebra generated by $\bfid(m)$ ($m\in \bbZ$).
\end{rem}
\begin{rem}We do not know whether the elements $a_m$ ($m \in \bbZ$) satisfy the Heisenberg relations. The construction of the Heisenberg subalgebra inside the affine Yangian will be left as a future work.
\end{rem}

\appendix
\section{General case}\label{appendixA}
We use the notation $(\ad X)^{(n)}(Y) = (\ad X)^{n}(Y) / n!$ for divided power operators.

\subsection{Yangian and braid group action}

Let $(a_{ij})_{i,j \in I}$ be a symmetrizable generalized Cartan matrix and fix integers $(d_i)_{i \in I}$ such that $(d_i a_{ij})_{i,j \in I}$ is symmetric. We denote by $\frakg$ the corresponding Kac--Moody Lie algebra. Then the Yangian $Y(\frakg)$ is defined to be generated by $x_{i,r}^{+}$, $x_{i,r}^{-}$, $h_{i,r}$ ($i \in I$, $r \in \mathbb{Z}_{\geq 0}$) with a parameter $\hbar \in \bbC$ subject to the relations:
\begin{gather*}
[h_{i,r}, h_{j,s}] = 0, \qquad \big[x_{i,r}^{+}, x_{j,s}^{-}\big] = \delta_{ij} h_{i, r+s}, \qquad \big[h_{i,0}, x_{j,r}^{\pm}\big] = \pm d_i a_{ij} x_{j,r}^{\pm},\\
\big[h_{i, r+1}, x_{j, s}^{\pm}\big] - \big[h_{i, r}, x_{j, s+1}^{\pm}\big] = \pm d_i a_{ij} \tfrac{\hbar}{2} \big\{h_{i, r}, x_{j, s}^{\pm}\big\},\\
\big[x_{i, r+1}^{\pm}, x_{j, s}^{\pm}\big] - \big[x_{i, r}^{\pm}, x_{j, s+1}^{\pm}\big] = \pm d_i a_{ij}\tfrac{\hbar}{2} \big\{x_{i, r}^{\pm}, x_{j, s}^{\pm}\big\},\\
\sum_{w \in \mathfrak{S}_{1 - a_{ij}}}\big[x_{i,r_{w(1)}}^{\pm}, \big[x_{i,r_{w(2)}}^{\pm}, \dots, \big[x_{i,r_{w(1 - a_{ij})}}^{\pm}, x_{j,s}^{\pm}\big]\dots\big]\big] = 0, \qquad i \neq j.
\end{gather*}
Set $x_{i}^{\pm} = x_{i,0}^{\pm}$, $h_{i} = h_{i,0}$. Then the standard Chevalley generators of $\frakg$ are identified with $d_i^{-1/2} x_i^+$, $d_i^{-1/2} x_i^-$, $d_i^{-1}h_i$.

Following \cite{MR3861718}, we define
\begin{gather*}
T_i = \exp\ad d_i^{-1/2}x_i^+ \exp\ad \big({-}d_i^{-1/2}x_i^-\big) \exp\ad d_i^{-1/2}x_i^+.
\end{gather*}
In the sequel, we put $e_i = d_i^{-1/2}x_i^+$, $f_i=d_i^{-1/2}x_i^-$.
\begin{prop}
The operators $\{T_i\}$ satisfy the braid relations. That is, we have
\begin{gather*}
T_i T_j = T_j T_i \quad \text{if $a_{ij}=0$},\qquad T_i T_j T_i = T_j T_i T_j \quad \text{if $a_{ij}=-1$},\\
T_i T_j T_i T_j = T_j T_i T_j T_i \quad \text{if $a_{ij}=-2$},\qquad T_i T_j T_i T_j T_i T_j = T_j T_i T_j T_i T_j T_i \quad \text{if $a_{ij}=-3$}.
\end{gather*}
\end{prop}
\begin{proof}The identities follow from
\begin{alignat*}{4}
& T_i(e_j) = e_j, \qquad && T_i(f_j) = f_j, \qquad && a_{ij}=0,& \\
& T_i T_j (e_i) = e_j, \qquad && T_i T_j (f_i) = f_j, \qquad && a_{ij}=-1,& \\
& T_i T_j T_i (e_j) = e_j, \qquad && T_i T_j T_i (f_j) = f_j, \qquad && a_{ij}=-2,& \\
& T_i T_j T_i T_j T_i (e_j) = e_j, \qquad && T_i T_j T_i T_j T_i (f_j) = f_j, \qquad && a_{ij}=-3 &
\end{alignat*}
as in the proof of Proposition~\ref{prop:braid}.
\end{proof}

\begin{prop}We have
\begin{gather*}
T_i\big(x_j^{\pm}\big) = \begin{cases}
-x_i^{\mp} & \text{if $i=j$},\\
\big({\pm} \ad d_i^{-1/2}x_i^{\pm}\big)^{(-a_{ij})}\big(x_j^{\pm}\big) & \text{if $a_{ij}<0$},\\
x_j^{\pm} & \text{if $a_{ij}=0$},
\end{cases}\qquad
T_i(h_j) = \begin{cases}
-h_i & \text{if $i=j$},\\
h_j - a_{ij} h_i & \text{if $a_{ij}<0$},\\
h_j & \text{if $a_{ij}=0$},
\end{cases}
\\
T_i\big(x_{j,1}^{\pm}\big) = \begin{cases}
-x_{i,1}^{\mp} + \tfrac{\hbar}{2}\big\{ h_i, x_i^{\mp} \big\} & \text{if $i=j$},\\
\big({\pm} \ad d_i^{-1/2} x_i^{\pm}\big)^{(-a_{ij})}\big(x_{j,1}^{\pm}\big) & \text{if $a_{ij} <0 $},\\
x_{j,1}^{\pm} & \text{if $a_{ij}=0$}.
\end{cases}
\end{gather*}
\end{prop}

\begin{proof} The formulas for $T_i\big(x_j^{\pm}\big)$, $T_i(h_j)$ are deduced from well-known formulas for the Chevalley generators. We produce a computation of $T_i\big(x_j^{-}\big)$ for $a_{ij}<0$ since the case $T_i\big(x_{j,1}^{-}\big)$ for $a_{ij}<0$ is verified in a very similar way. Put $m = -a_{ij}$. We have
\begin{gather}
T_i(x_{j}^{-}) = \exp\ad e_i \exp\ad(-f_i) \big(x_{j}^{-}\big) = \sum_{n=0}^{m} \dfrac{(-1)^n}{n!} \exp\ad e_i (\ad f_i)^n\big(x_{j}^{-}\big).\label{eq:n!}
\end{gather}
We can prove
\begin{gather}
\exp\ad e_i (\ad f_i)^n \big(x_j^-\big) = \sum_{k=0}^n \begin{binom}
n
k
\end{binom} \left(\prod_{l=k}^{n-1} (m-l)\right) (\ad f_i)^k\big(x_j^-\big) \label{eq:binom}
\end{gather}
by induction on $n$. Indeed, if we assume the assertion for $n$, we have
\begin{gather}
\exp\ad e_i (\ad f_i)^{n+1} \big(x_j^-\big) = \ad \left( \exp\ad e_i(f_i) \right) \big( \exp\ad e_i (\ad f_i)^{n} \big(x_j^-\big) \big)\label{eq:inductive}\\
\hphantom{\exp\ad e_i (\ad f_i)^{n+1} \big(x_j^-\big)}{} =\ad\big(f_i + d_i^{-1}h_i - e_i\big) \left( \sum_{k=0}^n
\begin{binom}
n
k
\end{binom} \left(\prod_{l=k}^{n-1} (m-l)\right) (\ad f_i)^k\big(x_j^-\big) \right). \nonumber
\end{gather}
Since we have
\begin{gather*}
\ad h_i (\ad f_i)^k\big(x_j^-\big) = d_i(-2k+m) (\ad f_i)^k\big(x_j^-\big)
\end{gather*}
and
\begin{gather*}
\ad e_i (\ad f_i)^k\big(x_j^-\big) = \sum_{s=0}^{k-1} (\ad f_i)^{k-s-1} \big(\ad d_i^{-1} h_i\big) (\ad f_i)^s\big(x_j^-\big) \\
\hphantom{\ad e_i (\ad f_i)^k\big(x_j^-\big)}{} = \sum_{s=0}^{k-1} (-2s+m) (\ad f_i)^{k-1}\big(x_j^-\big) = k(m-k+1) (\ad f_i)^{k-1}\big(x_j^-\big),
\end{gather*}
the coefficient of $(\ad f_i)^k\big(x_j^-\big)$ in (\ref{eq:inductive}) is
\begin{gather*}
\begin{binom}
n
{k-1}
\end{binom} \left(\prod_{l=k-1}^{n-1} (m-l)\right)
+
\begin{binom}
n
{k}
\end{binom} \left( \prod_{l=k}^{n-1} (m-l) \right) (m-2k)\\
\qquad {}-
\begin{binom}
n
{k+1}
\end{binom} \left( \prod_{l=k+1}^{n-1} (m-l) \right) (k+1)(m-k).
\end{gather*}
It is easy to see that this is equal to
\begin{gather*}
\begin{binom}
{n+1}
{k}
\end{binom} \left( \prod_{l=k}^{n} (m-l) \right).
\end{gather*}
Thus (\ref{eq:binom}) is proved. Then~(\ref{eq:n!}) is equal to
\begin{gather*}
 \sum_{n=0}^{m} \dfrac{(-1)^n}{n!} \sum_{k=0}^n \begin{binom}
n
k
\end{binom} \left(\prod_{l=k}^{n-1} (m-l)\right) (\ad f_i)^k\big(x_j^-\big) \\
 \qquad {} = \sum_{k=0}^m \dfrac{(-1)^k}{k!} (\ad f_i)^k\big(x_j^-\big) \times \left( \sum_{n=k}^{m} (-1)^{n-k} \begin{binom}
{m-k}
{n-k}
\end{binom} \right) = \dfrac{(-1)^m}{m!} (\ad f_i)^m\big(x_j^-\big).
\end{gather*}

The formula for $T_i\big(x_{i,1}^-\big)$ is proved in a way similar to Proposition~\ref{prop:formula2}. We use
\begin{gather*}
\exp\ad e_i\big(x_{i,1}^+\big) = x_{i,1}^+ - d_i^{1/2} \hbar \big(x_i^+\big)^2, \\
\exp\ad e_i \big(x_{i,1}^-\big) = x_{i,1}^- + d_i^{-1/2} h_{i,1} - x_{i,1}^+ - \tfrac{\hbar}{2}\big\{h_i,x_i^+\big\} + d_i^{1/2} \hbar \big(x_i^+\big)^2, \\
\exp\ad e_i \big(\big\{h_i,x_i^+\big\}\big) = \big\{h_i,x_i^+\big\} - 4d_i^{1/2}\big(x_i^+\big)^2, \\
\exp\ad e_i \big(\big(x_i^+\big)^2\big) = \big(x_i^+\big)^2,
\\
\exp\ad (-f_i)\big(x_{i,1}^+\big)= x_{i,1}^+ + d_i^{-1/2} h_{i,1} - x_{i,1}^- - \tfrac{\hbar}{2}\big\{h_i,x_i^-\big\} + d_i^{1/2} \hbar \big(x_i^-\big)^2, \\
\exp\ad (-f_i)\big(x_{i,1}^-\big) = x_{i,1}^- - d_i^{1/2} \hbar \big(x_i^-\big)^2, \\
\exp\ad (-f_i)(h_{i,1}) = h_{i,1} - 2 d_i^{1/2} x_{i,1}^- - d_i^{1/2} \hbar \big\{h_i,x_i^-\big\} + 3 d_i \hbar \big(x_i^-\big)^2, \\
\exp\ad (-f_i)\big(\big\{h_{i},x_i^+\big\}\big) = \big\{h_i,x_i^+\big\} - 3 \big\{h_i,x_i^-\big\} - 2d_i^{1/2}\big\{x_i^+,x_i^-\big\} \\
\hphantom{\exp\ad (-f_i)\big(\big\{h_{i},x_i^+\big\}\big) =}{} + 2d_i^{-1/2} h_i^2 + 4d_i^{1/2}\big(x_i^-\big)^2, \\
\exp\ad (-f_i)\big(\big(x_i^+\big)^2\big) = \big(x_i^+\big)^2 + d_i^{-1} h_i^2 + \big(x_i^-\big)^2 + d_i^{-1/2} \big\{h_i,x_i^+\big\} \\
\hphantom{\exp\ad (-f_i)\big(\big(x_i^+\big)^2\big) =}{} - d_i^{-1/2} \big\{h_i,x_i^-\big\} - \big\{x_i^+,x_i^-\big\}.
\end{gather*}

As we mentioned, $T_i\big(x_{j,1}^-\big)$ for $a_{ij}<0$ is computed by replacing $x_j^-$ with $x_{j,1}^-$ in the argument for $T_i\big(x_{j}^-\big)$. Then the formulas for $T_i\big(x_{j,1}^+\big)$ are obtained from those for $T_i\big(x_{j,1}^-\big)$ by applying~$\omega$.
\end{proof}

\begin{rem} In this appendix we impose no further condition on $(a_{ij})_{i,j \in I}$, and hence on $\frakg$, to study the braid group action. However the defining relations of $Y(\frakg)$ given here may not produce a correct definition of Yangian for some generalized Cartan matrix. For example, it is known that for $Y(\hat{\mathfrak{sl}}_2)$ the defining relations should be modified as in \cite[Definition~5.1]{MR3898327} or \cite[Section~1.2]{MR3850574} (In~\cite{MR3898327}, some relations are missing). In \cite[equation~(2.1)]{GRW}, the authors suggest the following condition:
\begin{gather*}
\text{for all $i$ and $ j$, if $a_{ij} \leq -2$ then $a_{ji}=-1$ holds}.
\end{gather*}
See \cite[Section~2, Lemma~4.2, Remark~5.15]{GRW} for an explanation on some evidences.
\end{rem}

\subsection{Compatibility with the coproduct}

Assume that $\frakg$ is a Kac--Moody Lie algebra of finite or affine type except for $A_1^{(1)}$ and $A_2^{(2)}$. Then the coproduct $\Delta$ on $Y(\frakg)$ is well-defined by the same formula as in Theorem~\ref{thm:definition_of_coproduct} \cite[Definition~4.6, Theorem~4.9, Proposition~5.18]{MR3861718}. Let us prove the main result of this appendix.
\begin{prop}\label{prop:compatibility_coproduct2} We have $\Delta \circ T_i = (T_i \otimes T_i) \circ \Delta$.
\end{prop}
\begin{proof}Lemma~\ref{lem:Omega} holds in a general situation and we use it. It is enough to prove $\Delta T_i\big(x_{j,1}^+\big) = (T_i \otimes T_i)\Delta\big(x_{j,1}^+\big)$ for $a_{ij} < 0$ since the proofs concerning the other generators are the same as in the proof of Proposition~\ref{prop:compatibility_coproduct}. We claim that the both sides coincide with
\begin{gather*}
\square\,T_i(x_{j,1}^+)-\hbar \big( \big[1 \otimes T_i\big(x_j^+\big),\Omega_+\big] - d_i^{1/2}x_i^+ \otimes (\ad e_i)^{(-a_{ij}-1)}\big(x_j^+\big) \big).
\end{gather*}
The left-hand side is
\begin{gather*}
\Delta \big( (\ad e_i)^{(-a_{ij})} \big(x_{j,1}^+\big)\big) = (\ad \square(e_i))^{(-a_{ij})} \big(\square\big(x_{j,1}^+\big) - \hbar\big[1 \otimes x_j^+,\Omega_+\big]\big) \\
= \square\,T_i\big(x_{j,1}^+\big) - \hbar\bigg( \big[(\ad \square(e_i))^{(-a_{ij})}\big(1 \otimes x_j^+\big),\Omega_+\big] \\
\quad{} + \sum_{n=1}^{-a_{ij}} \big[(\ad \square(e_i))^{(-a_{ij}-n)}\big(1 \otimes x_j^+\big),(\ad \square(e_i))^{(n)}(\Omega_+)\big]\bigg) \\
= \square\,T_i\big(x_{j,1}^+\big) - \hbar\bigg( \big[1 \otimes T_i\big(x_j^+\big),\Omega_+\big] + \sum_{n=1}^{-a_{ij}} \big[1 \otimes (\ad e_i)^{(-a_{ij}-n)}\big(x_j^+\big),(\ad \square(e_i))^{(n)}(\Omega_+)\big]\bigg).
\end{gather*}
We show
\begin{gather*}
\sum_{n=1}^{-a_{ij}} \big[1 \otimes (\ad e_i)^{(-a_{ij}-n)}\big(x_j^+\big),(\ad \square(e_i))^{(n)}(\Omega_+)\big] = - d_i^{1/2} x_i^+ \otimes (\ad e_i)^{(-a_{ij}-1)}\big(x_j^+\big).
\end{gather*}
Since we have
\begin{gather*}
(\ad \square(e_i))^{(n)}(\Omega_+) = \begin{cases}
-d_i^{-1/2} x_i^+ \otimes h_i & \text{if $n=1$,}\\
x_i^+ \otimes x_i^+ & \text{if $n=2$,}\\
0& \text{if $n \geq 3$,}
\end{cases}
\end{gather*}
we see
\begin{gather*}
\begin{split}&
\sum_{n=1}^{-a_{ij}} \big[1 \otimes (\ad e_i)^{(-a_{ij}-n)}\big(x_j^+\big),(\ad \square(e_i))^{(n)}(\Omega_+)\big] \\
& \qquad{} = \big[1 \otimes x_j^+, -d_i^{-1/2}x_i^+ \otimes h_i\big] = - d_i^{1/2} x_i^+ \otimes x_j^+\end{split}
\end{gather*}
if $a_{ij}=-1$, and
\begin{gather*}
\sum_{n=1}^{-a_{ij}} \big[1 \otimes (\ad e_i)^{(-a_{ij}-n)}\big(x_j^+\big),(\ad \square(e_i))^{(n)}(\Omega_+)\big] \\
\qquad{} = \big[1 \otimes (\ad e_i)^{(-a_{ij}-1)}\big(x_j^+\big), -d_i^{-1/2}x_i^+ \otimes h_i\big] + \big[1 \otimes (\ad e_i)^{(-a_{ij}-2)}\big(x_j^+\big),x_i^+ \otimes x_i^+\big] \\
\qquad{}= d_i^{-1/2}\big( 2d_i(-a_{ij}-1) + d_ia_{ij} \big) x_i^+ \otimes (\ad e_i)^{(-a_{ij}-1)}\big(x_j^+\big) \\
\qquad\quad{}- d_i^{1/2} ( -a_{ij}-1 ) x_i^+ \otimes (\ad e_i)^{(-a_{ij}-1)}\big(x_j^+\big) \\
\qquad {}= - d_i^{1/2} x_i^+ \otimes (\ad e_i)^{(-a_{ij}-1)}\big(x_j^+\big)
\end{gather*}
if $a_{ij} \leq -2$, hence the claim holds. The right-hand side is
\begin{gather*}
(T_i \otimes T_i)\big(\square\big(x_{j,1}^+\big) - \hbar\big[1 \otimes x_j^+,\Omega_+\big]\big) = \square\,T_i\big(x_{j,1}^+\big) - \hbar \big[1 \otimes T_i\big(x_j^+\big),\big(T_i \otimes T_i\big) \Omega_+\big]\\
\qquad{} = \square\,T_i\big(x_{j,1}^+\big) - \hbar \big[1 \otimes T_i\big(x_j^+\big),\Omega_+ + x_i^+ \otimes x_i^- - x_i^- \otimes x_i^+\big]\\
\qquad{} = \square\,T_i\big(x_{j,1}^+\big) - \hbar \big( \big[1 \otimes T_i\big(x_j^+\big),\Omega_+\big] + x_i^+ \otimes \big[T_i\big(x_j^+\big), x_i^-\big] - x_i^- \otimes \big[T_i(x_j^+),x_i^+\big]\big).
\end{gather*}
Then the desired identity follows from
\begin{gather*}
\big[T_i(x_j^+), x_i^-\big] = \left[\dfrac{1}{(-a_{ij})!} (\ad e_i)^{-a_{ij}}\big(x_j^+\big),d_i^{1/2} f_i\right]\\
\hphantom{\big[T_i(x_j^+), x_i^-\big]}{} = \dfrac{1}{(-a_{ij})!} d_i^{1/2} \sum_{n=0}^{-a_{ij}-1} (\ad e_i)^{-a_{ij}-n-1} \big(\ad d_i^{-1} h_i\big) (\ad e_i)^{n} \big(x_j^+\big) \\
\hphantom{\big[T_i(x_j^+), x_i^-\big]}{} = \dfrac{1}{(-a_{ij})!} d_i^{1/2} \sum_{n=0}^{-a_{ij}-1} (2n+a_{ij}) (\ad e_i)^{-a_{ij}-1} \big(x_j^+\big) \\
\hphantom{\big[T_i(x_j^+), x_i^-\big]}{} = \dfrac{1}{(-a_{ij})!} d_i^{1/2} a_{ij} (\ad e_i)^{-a_{ij}-1} \big(x_j^+\big)
 = -d_i^{1/2}(\ad e_i)^{(-a_{ij}-1)}\big(x_j^+\big)
\end{gather*}
and
\begin{gather*}
\big[T_i(x_j^+), x_i^+\big] = \left[\dfrac{1}{(-a_{ij})!} (\ad e_i)^{-a_{ij}}\big(x_j^+\big),d_i^{1/2} e_i\right] = -\dfrac{1}{(-a_{ij})!} d_i^{1/2} (\ad e_i)^{-a_{ij}+1}\big(x_j^+\big)=0.\!\!\!\!\tag*{\qed}
\end{gather*} \renewcommand{\qed}{}
\end{proof}

\subsection*{Acknowledgments}
The author would like to thank Yoshihisa Saito for suggesting him to study braid group action on affine Yangian. Discussions with him improved contents of this paper. He also thanks Nicolas Guay for telling him the reference \cite{MR1745710} and the referees for helpful comments. This work was supported by JSPS KAKENHI Grant Number 26287004, 17H06127, 18K13390, and The Kyoto University Foundation.

\pdfbookmark[1]{References}{ref}
\LastPageEnding


\begin{thebibliography}{99}
\footnotesize\itemsep=0pt

\bibitem{MR1301623}
Beck J., Braid group action and quantum affine algebras, \href{https://doi.org/10.1007/BF02099423}{\textit{Comm. Math.
 Phys.}} \textbf{165} (1994), 555--568, \href{https://arxiv.org/abs/hep-th/9404165}{arXiv:hep-th/9404165}.

\bibitem{MR3850574}
Bershtein M., Tsymbaliuk A., Homomorphisms between different quantum toroidal
 and affine {Y}angian algebras, \href{https://doi.org/10.1016/j.jpaa.2018.05.003}{\textit{J.~Pure Appl. Algebra}} \textbf{223}
 (2019), 867--899, \href{https://arxiv.org/abs/1512.09109}{arXiv:1512.09109}.

\bibitem{MR1745710}
Ding J., Khoroshkin S., Weyl group extension of quantized current algebras,
 \href{https://doi.org/10.1007/BF01237177}{\textit{Transform. Groups}} \textbf{5} (2000), 35--59,
 \href{https://arxiv.org/abs/math.QA/9804139}{arXiv:math.QA/9804139}.

\bibitem{MR2199856}
Guay N., Cherednik algebras and {Y}angians, \href{https://doi.org/10.1155/IMRN.2005.3551}{\textit{Int. Math. Res. Not.}}
 \textbf{2005} (2005), 3551--3593.

\bibitem{MR2323534}
Guay N., Affine {Y}angians and deformed double current algebras in type~{A},
 \href{https://doi.org/10.1016/j.aim.2006.08.007}{\textit{Adv. Math.}} \textbf{211} (2007), 436--484.

\bibitem{MR3861718}
Guay N., Nakajima H., Wendlandt C., Coproduct for {Y}angians of affine
 {K}ac--{M}oody algebras, \href{https://doi.org/10.1016/j.aim.2018.09.013}{\textit{Adv. Math.}} \textbf{338} (2018), 865--911,
 \href{https://arxiv.org/abs/1701.05288}{arXiv:1701.05288}.

\bibitem{GRW}
Guay N., Regelskis V., Wendlandt C., Vertex representations for {Y}angians of
 {K}ac--{M}oody algebras, \href{https://arxiv.org/abs/1804.04081}{arXiv:1804.04081}.

\bibitem{MR1081014}
Kirillov A.N., Reshetikhin N., {$q$}-{W}eyl group and a multiplicative formula
 for universal {$R$}-matrices, \href{https://doi.org/10.1007/BF02097710}{\textit{Comm. Math. Phys.}} \textbf{134} (1990),
 421--431.

\bibitem{MR3898327}
Kodera R., Affine {Y}angian action on the {F}ock space, \href{https://doi.org/10.4171/PRIMS/55-1-6}{\textit{Publ. Res.
 Inst. Math. Sci.}} \textbf{55} (2019), 189--234, \href{https://arxiv.org/abs/1506.01246}{arXiv:1506.01246}.

\bibitem{kodera_evaluation}
Kodera R., On Guay's evaluation map for affine {Y}angians, \href{https://arxiv.org/abs/1806.09884}{arXiv:1806.09884}.

\bibitem{MR1120927}
Levendorskii S.Z., Soibelman Y.S., Some applications of the quantum {W}eyl
 groups, \href{https://doi.org/10.1016/0393-0440(90)90013-S}{\textit{J.~Geom. Phys.}} \textbf{7} (1990), 241--254.

\bibitem{MR1035415}
Lusztig G., Canonical bases arising from quantized enveloping algebras,
 \href{https://doi.org/10.2307/1990961}{\textit{J.~Amer. Math. Soc.}} \textbf{3} (1990), 447--498.

\bibitem{MR1013053}
Lusztig G., Finite-dimensional {H}opf algebras arising from quantized universal
 enveloping algebra, \href{https://doi.org/10.2307/1990988}{\textit{J.~Amer. Math. Soc.}} \textbf{3} (1990), 257--296.

\bibitem{MR1066560}
Lusztig G., Quantum groups at roots of~{$1$}, \href{https://doi.org/10.1007/BF00147341}{\textit{Geom. Dedicata}}
 \textbf{35} (1990), 89--113.

\bibitem{MR1227098}
Lusztig G., Introduction to quantum groups, \textit{Progress in Mathematics},
 Vol.~110, \href{https://doi.org/10.1007/978-0-8176-4717-9}{Birkh\"{a}user Boston, Inc.}, Boston, MA, 1993.

\bibitem{MR1265471}
Saito Y., P{BW} basis of quantized universal enveloping algebras, \href{https://doi.org/10.2977/prims/1195166130}{\textit{Publ.
 Res. Inst. Math. Sci.}} \textbf{30} (1994), 209--232.

\bibitem{MR1724950}
Uglov D., Symmetric functions and the {Y}angian decomposition of the {F}ock and
 basic modules of the affine {L}ie algebra {$\widehat{\mathfrak{sl}}_N$}, in
 Quantum Many-Body Problems and Representation Theory, \textit{MSJ Mem.},
 Vol.~1, Math. Soc. Japan, Tokyo, 1998, 183--241, \href{https://arxiv.org/abs/q-alg/9705010}{arXiv:q-alg/9705010}.

\bibitem{MR1818101}
Varagnolo M., Quiver varieties and {Y}angians, \href{https://doi.org/10.1023/A:1007674020905}{\textit{Lett. Math. Phys.}}
 \textbf{53} (2000), 273--283, \href{https://arxiv.org/abs/math.QA/0005277}{arXiv:math.QA/0005277}.

\end{thebibliography}
\end{document}